\documentclass[12pt, a4paper, reqno]{amsart}
\usepackage{amsmath}
\usepackage{amsfonts}
\usepackage{amssymb}
\usepackage{color}
\usepackage{comment,graphicx}
\usepackage[T1]{fontenc}
\usepackage{url}
\usepackage{ae}

\def\phi{\varphi}
\def\rho{\varrho}
\def\epsilon{\varepsilon}
\numberwithin{equation}{section}
\theoremstyle{plain}
\newtheorem{theorem}[equation]{Theorem}
\newtheorem{lemma}[equation]{Lemma}

\theoremstyle{definition}
\newtheorem{definition}[equation]{Definition}

\theoremstyle{remark}
\newtheorem{remark}[equation]{Remark}
\newtheorem{example}[equation]{Example}

\renewcommand{\le}{\leqslant}
\renewcommand{\ge}{\geqslant}
\renewcommand{\leq}{\leqslant}
\renewcommand{\geq}{\geqslant}

\begin{document}
	\title[Embeddings in Herz-type Sobolev spaces on domains]{Embedding theorems for Herz-Type Sobolev spaces on domains}
	\author{Douadi Drihem}
	\address{Douadi Drihem\\
		Laboratory of Functional Analysis and Geometry of Spaces, Faculty of
		Mathematics and Informatics, Department of Mathematics,\\
		M'sila University, PO Box 166 Ichebelia, M'sila 28000, Algeria}
	\email{douadidr@yahoo.fr, douadi.drihem@univ-msila.dz}
	\date{\today }
	\maketitle

	\begin{abstract}
		In this paper, we further investigate the properties of Herz-type Sobolev spaces on domains, which unify and generalize the classical Sobolev spaces. In particular, we establish several Sobolev-type embedding theorems for these spaces. Moreover, by means of a suitable example, we show that some of the assumptions imposed on the parameters are optimal.
		
		
	\end{abstract}
\section{Introduction}

Sobolev spaces are among the most important function spaces in modern
analysis, with fundamental applications to harmonic analysis and partial
differential equations. We refer the reader to the monographs \cite%
{AdamsFournier03}, \cite{Burenkov}, \cite{Ma85} and \cite{Tartar} for
comprehensive accounts, historical remarks, and further references
concerning Sobolev spaces.

Herz spaces, introduced in \cite{Herz68}, also play an important role in harmonic analysis, particularly in the study of singular integral operators and Fourier multipliers. Interest in these spaces stems not only from their theoretical significance but also from their broad range of applications in analysis; see, for example, \cite{BS85}, \cite{Dr-Herz-Heat}, \cite{Fi-We08}, \cite{Rag09}, and \cite{Tu11}. For recent developments and further important results concerning Herz spaces, we refer the reader to \cite{RS}, \cite{ZYZ}, and the monograph \cite{LYH}.

Motivated by Sobolev and Herz spaces, the author in \cite{Dr-Sobolev} introduced the so-called Herz-type Sobolev spaces on domains and established several fundamental properties of these spaces, including Sobolev embedding theorems. Similar function spaces over $\mathbb{R}^{n}$ were previously introduced by Lu and Yang in \cite{LuYang97}, where applications to partial differential equations were also provided.

In this paper, we present further results extending those obtained in \cite{Dr-Sobolev}. In particular, the Sobolev embedding theorems established in \cite{Dr-Sobolev} were proved under certain additional assumptions. We improve these embedding results and discuss the optimality of the corresponding conditions. Furthermore, we establish several new embedding theorems for Herz-type Sobolev spaces.

The paper is organized as follows. In Section 2, we present the necessary preliminaries, fix the notation, and recall some basic facts concerning Herz spaces, including interpolation inequalities and density results.

In Section 3, we establish the fundamental properties of Herz-type Sobolev spaces, in analogy with the classical Sobolev spaces. In Section 4, we develop the main technical tools required for the proofs of our principal results. Finally, in Section 5, we prove Sobolev embedding theorems for Herz-type Sobolev spaces. We also present a suitable example showing that some of the assumptions on the parameters are optimal.

We believe that the theory of Herz-type Sobolev spaces is still at an early stage of development. We hope that further properties and applications of these spaces will be investigated in future work.

The results established in this paper can be naturally extended to Lorentz-Herz-type Sobolev spaces by replacing the Lebesgue spaces in the definition of Herz spaces with the corresponding Lorentz spaces. However, to avoid unnecessary technical complications and for the convenience of the reader, we restrict our exposition to Herz-type Sobolev spaces.

\subsection{Notation and conventions}

As usual, $\mathbb{R}^{n}$ denotes the $n$-dimensional Euclidean space, $%
\mathbb{N}$ the set of all natural numbers, and $\mathbb{N}_{0}=\mathbb{N}%
\cup \{0\}$. The symbol $\mathbb{Z}$ stands for the set of all integers. For 
$u>0$, we define 
$
R(u)=\left\{ x\in \mathbb{R}^{n}:\frac{u}{2}\leq |x|<u\right\}.$
For $x\in \mathbb{R}^{n}$ and $r>0$, we denote by $B(x,r)$ the open ball in $%
\mathbb{R}^{n}$ centered at $x$ with radius $r$. 
If $1\leq p<\infty$ and 
$\frac{1}{p}+\frac{1}{p^{\prime }}=1,$
then $p^{\prime }$ is called the conjugate exponent of $p$.

\medskip

We denote by $|\Omega|$ the $n$-dimensional Lebesgue measure of a measurable
set $\Omega \subseteq \mathbb{R}^{n}$. For any measurable set $\Omega
\subseteq \mathbb{R}^{n}$, the Lebesgue space $L^{p}(\Omega)$, $0<p\leq
\infty$, consists of all measurable functions $f$ such that 
\begin{equation*}
\|f\|_{L^{p}(\Omega)} =\left( \int_{\Omega}|f(x)|^{p}\,dx
\right)^{1/p}<\infty, \qquad 0<p<\infty,
\end{equation*}
and 
\begin{equation*}
\big\Vert f\big\Vert_{L^{\infty }(\Omega )}=\underset{x\in \Omega }{\text{%
		ess-sup}}\left\vert f(x)\right\vert <\infty .
\end{equation*}
If $\Omega=\mathbb{R}^{n}$, we simply write 
\begin{equation*}
\|f\|_{L^{p}(\mathbb{R}^{n})}=\|f\|_{p}. 
\end{equation*}

\medskip

If $G\subset \mathbb{R}^{n}$ is nonempty, we denote by $\overline{G}$ the
closure of $G$ in $\mathbb{R}^{n}$. Let $\Omega \subseteq \mathbb{R}^{n}$ be
open. We write $G\Subset \Omega$ if $\overline{G}\subset \Omega$ and $G$ is
compact in $\mathbb{R}^{n}$. For any nonnegative integer $m$, let $%
C^{m}(\Omega)$ denote the vector space of all functions $f$ such that $f$
and all their partial derivatives $D^{\beta}f$ of order $|\beta|\leq m$ are
continuous on $\Omega$. We set 
\begin{equation*}
C^{0}(\Omega)=C(\Omega), \qquad C^{\infty}(\Omega)=\bigcap_{m\geq 0}
C^{m}(\Omega).
\end{equation*}
Finally, we denote by $C_{c}(\Omega)$ the space of all continuous functions
on $\Omega$ with compact support in $\Omega$.

\section{Herz spaces}

In this section, we present some fundamental properties of Herz spaces. Let $%
k\in \mathbb{Z}$. For convenience, we set 
\begin{equation*}
B_{k}=B(0,2^{k}), \qquad R_{k}=B_{k}\setminus B_{k-1}.
\end{equation*}
We denote by $\chi_{k}$ the characteristic function of the set $R_{k}$.

\begin{definition}
	\label{def-inh-Herz} Let $\alpha \in \mathbb{R}$ and $1\leq p,q\leq \infty$.
	
	\medskip
	
	\noindent \textnormal{(i)} The homogeneous Herz space $\dot{K}%
	_{p}^{\alpha,q}(\mathbb{R}^{n})$ is defined as the set of all functions $%
	f\in L_{\mathrm{loc}}^{p}(\mathbb{R}^{n}\setminus\{0\})$ such that 
	\begin{equation*}
	\|f\|_{\dot{K}_{p}^{\alpha,q}(\mathbb{R}^{n})} = \left(
	\sum_{k=-\infty}^{\infty} 2^{k\alpha q} \|f\chi_{k}\|_{p}^{q} \right)^{1/q}
	<\infty,
	\end{equation*}
	with the usual modifications when $p=\infty$ and/or $q=\infty$.
	
	\medskip
	
	\noindent \textnormal{(ii)} The nonhomogeneous Herz space $K_{p}^{\alpha,q}(%
	\mathbb{R}^{n})$ is defined as the set of all functions $f\in L_{\mathrm{loc}%
	}^{p}(\mathbb{R}^{n})$ such that 
	\begin{equation*}
	\|f\|_{K_{p}^{\alpha,q}(\mathbb{R}^{n})} = \|f\chi_{B_{0}}\|_{p} + \left(
	\sum_{k=1}^{\infty} 2^{k\alpha q} \|f\chi_{k}\|_{p}^{q} \right)^{1/q}
	<\infty,
	\end{equation*}
	with the usual modifications when $p=\infty$ and/or $q=\infty$.
\end{definition}

\begin{remark}
	The spaces $K_{p}^{\alpha,q}(\mathbb{R}^{n})$ and $\dot{K}_{p}^{\alpha,q}(%
	\mathbb{R}^{n})$ are Banach spaces. If $\alpha=0$ and $1\leq p=q\leq \infty$%
	, then both $\dot{K}_{p}^{0,p}(\mathbb{R}^{n})$ and $K_{p}^{0,p}(\mathbb{R}%
	^{n})$ coincide with the Lebesgue space $L^{p}(\mathbb{R}^{n})$. Moreover, 
	\begin{equation*}
	\dot{K}_{p}^{\alpha,p}(\mathbb{R}^{n}) = L^{p}(\mathbb{R}^{n},|\cdot|^{%
		\alpha p}),
	\end{equation*}
	that is, the weighted Lebesgue space endowed with the norm 
	\begin{equation*}
	\|f\|_{L^{p}(\mathbb{R}^{n},|\cdot|^{\alpha p})} = \left( \int_{\mathbb{R}%
		^{n}} |f(x)|^{p}|x|^{\alpha p}\,dx \right)^{1/p}.
	\end{equation*}
	
	Let $0<p,q\leq \infty$ and $\alpha>0$. Then 
	\begin{equation*}
	\dot{K}_{p}^{\alpha,q}(\mathbb{R}^{n})\cap L^{p}(\mathbb{R}^{n}) =
	K_{p}^{\alpha,q}(\mathbb{R}^{n}),
	\end{equation*}
	in the sense of equivalent quasi-norms. Furthermore, 
	\begin{equation*}
	\|f\|_{K_{p}^{\alpha,q}(\mathbb{R}^{n})} \approx \left( \sum_{k=0}^{\infty}
	2^{k\alpha q} \|f\chi_{\widehat{R}_{k}}\|_{p}^{q} \right)^{1/q},
	\end{equation*}
	where $\widehat{R}_{k}=R_{k}$ for $k\in \mathbb{N}$ and $\widehat{R}%
	_{0}=B_{0}$.
	
	Let $0<p\leq \infty$, $0<q_{1}\leq q_{2}\leq \infty$, and $\alpha\in\mathbb{R}$.
	Then the following continuous embeddings hold:
	\begin{equation}
	\dot{K}_{p}^{\alpha,q_{1}}(\mathbb{R}^{n})
	\hookrightarrow
	\dot{K}_{p}^{\alpha,q_{2}}(\mathbb{R}^{n}),
	\qquad
	K_{p}^{\alpha,q_{1}}(\mathbb{R}^{n})
	\hookrightarrow
	K_{p}^{\alpha,q_{2}}(\mathbb{R}^{n}).
	\label{herz-emb}
	\end{equation}
	
	Moreover, if $\alpha_{2}\leq \alpha_{1}$ and $0<q\leq \infty$, then
	\begin{equation}
	K_{p}^{\alpha_{1},q}(\mathbb{R}^{n})
	\hookrightarrow
	K_{p}^{\alpha_{2},q}(\mathbb{R}^{n}).
	\label{herz-emb1}
	\end{equation}
\end{remark}

A detailed discussion of the properties of these spaces may be found in the
monograph \cite{LYH08}, the papers \cite{LuYang1.95}, \cite{LuYang2.95}, 
\cite{Hernandez1998}, and the references therein.

If $\Omega \subset \mathbb{R}^{n}$ is open and $f:\Omega \rightarrow \mathbb{%
	R}$ is measurable, then we write $f\in \dot{K}_{p}^{\alpha,q}(\Omega)$
whenever $f\chi_{\Omega}\in \dot{K}_{p}^{\alpha,q}(\mathbb{R}^{n})$, and we
define 
\begin{equation*}
\|f\|_{\dot{K}_{p}^{\alpha,q}(\Omega)} = \|f\chi_{\Omega}\|_{\dot{K}%
	_{p}^{\alpha,q}(\mathbb{R}^{n})}.
\end{equation*}
Similarly, we define 
\begin{equation*}
\|f\|_{K_{p}^{\alpha,q}(\Omega)} = \|f\chi_{\Omega}\|_{K_{p}^{\alpha,q}(%
	\mathbb{R}^{n})}.
\end{equation*}

Various important results have been established in the spaces $K_{p}^{\alpha
	,q}(\mathbb{R}^{n})$ and $\dot{K}_{p}^{\alpha ,q}(\mathbb{R}^{n})$ under
suitable assumptions on $\alpha $, $p$, and $q$. The conditions 
\begin{equation*}
-\frac{n}{p}<\alpha <n\left( 1-\frac{1}{p}\right) ,\qquad 1<p<\infty,
\end{equation*}%
play a crucial role in the study of the boundedness of classical operators
on $K_{p}^{\alpha ,q}(\mathbb{R}^{n})$ and $\dot{K}_{p}^{\alpha ,q}(\mathbb{R%
}^{n})$. This fact was first observed by Li and Yang \cite{LiYang96} in
their proof of the boundedness of the Hardy--Littlewood maximal operator. As
usual, we define 
\begin{equation*}
\mathcal{M}(f)(x)=\sup_{Q}\frac{1}{|Q|}\int_{Q}|f(y)|\,dy,\qquad f\in L_{%
	\mathrm{loc}}^{1}(\mathbb{R}^{n}),
\end{equation*}%
where the supremum is taken over all cubes $Q$ with sides parallel to the
coordinate axes and containing $x$. We also set 
\begin{equation*}
\mathcal{M}_{t}(f)=\big(\mathcal{M}(|f|^{t})\big)^{1/t},\qquad 0<t<\infty .
\end{equation*}

\begin{lemma}
	\label{Maximal-Inq} Let $1<p<\infty$ and $1\leq q\leq \infty$. If $f\in L_{%
		\mathrm{loc}}^{1}(\mathbb{R}^{n})$ and 
	\begin{equation*}
	-\frac{n}{p}<\alpha<n\left(1-\frac{1}{p}\right),
	\end{equation*}
	then 
	\begin{equation*}
	\|\mathcal{M}(f)\|_{\dot{K}_{p}^{\alpha,q}(\mathbb{R}^{n})} \leq c\|f\|_{%
		\dot{K}_{p}^{\alpha,q}(\mathbb{R}^{n})},
	\end{equation*}
	and 
	\begin{equation*}
	\|\mathcal{M}(f)\|_{K_{p}^{\alpha,q}(\mathbb{R}^{n})} \leq
	c\|f\|_{K_{p}^{\alpha,q}(\mathbb{R}^{n})}.
	\end{equation*}
\end{lemma}

Let $V_{\alpha,p,q}$ denote the set of all $(\alpha,p,q)\in \mathbb{R}\times
[1,\infty]^{2}$ such that:

\begin{itemize}
	\item $\alpha<n-\dfrac{n}{p}$, $1\leq p\leq \infty$, and $1\leq q\leq \infty$%
	;
	
	\item $\alpha=n-\dfrac{n}{p}$, $1\leq p\leq \infty$, and $q=1$.
\end{itemize}

The next lemma gives a necessary and sufficient condition on the parameters $%
\alpha$, $p$, and $q$ ensuring that 
\begin{equation*}
\langle T_{f},\varphi\rangle = \int_{\Omega}f(x)\varphi(x)\,dx, \qquad
\varphi\in \mathcal{D}(\Omega),
\end{equation*}
defines a regular distribution $T_{f}\in \mathcal{D}^{\prime }(\Omega)$ for
every $f\in \dot{K}_{p}^{\alpha,q}(\Omega)$.

\begin{lemma}
	\label{L1loc} Let $\Omega \subset \mathbb{R}^{n}$ be open with $0\in \Omega$%
	, and let $1\leq p,q\leq \infty$ and $\alpha\in \mathbb{R}$. Then 
	\begin{equation*}
	\dot{K}_{p}^{\alpha,q}(\Omega) \hookrightarrow L_{\mathrm{loc}}^{1}(\Omega)
	\end{equation*}
	if and only if $(\alpha,p,q)\in V_{\alpha,p,q}$.
\end{lemma}

\begin{remark}
	\label{L1loc1}
	For the proof of Lemma \ref{L1loc}, see \cite{Dr-Sobolev}. Let $\Omega
	\subset \mathbb{R}^{n}$ be open, $1\leq p,q\leq \infty$, and $\alpha\in 
	\mathbb{R}$. Then clearly 
	\begin{equation*}
	K_{p}^{\alpha,q}(\Omega) \hookrightarrow L_{\mathrm{loc}}^{1}(\Omega).
	\end{equation*}
\end{remark}

\begin{remark}
	In general, if $0\notin \Omega$, then the set $V_{\alpha,p,q}$ is not
	optimal. Consequently, Lemma \ref{L1loc} and Remark 	\ref{L1loc1} justify the definition of weak derivatives for functions belonging to Herz spaces.
\end{remark}

\begin{theorem}
	\label{dense} Let $\Omega \subset \mathbb{R}^{n}$ be an open set, $1\leq
	p<\infty$, and $1\leq q<\infty$.
	
	\medskip
	
	\noindent \textnormal{(i)} If $\alpha>-\dfrac{n}{p}$, then $C_{c}(\Omega)$
	is dense in $\dot{K}_{p}^{\alpha,q}(\Omega)$.
	
	\medskip
	
	\noindent \textnormal{(ii)} If $\alpha\in \mathbb{R}$, then $C_{c}(\Omega)$
	is dense in $K_{p}^{\alpha,q}(\Omega)$.
\end{theorem}

\begin{proof}
	The proof of \textnormal{(i)} is given in \cite{Dr-Sobolev}. Assertion 
	\textnormal{(ii)} can be proved similarly.
\end{proof}

We now present an interpolation inequality for Herz spaces.

\begin{lemma}
	\label{interpolation2} Let $\Omega \subset \mathbb{R}^{n}$ be an open set, $%
	0<p_{0},p_{1},q_{0},q_{1}\leq \infty$, and $\alpha_{0},\alpha_{1}\in \mathbb{%
		R}$. Define 
	\begin{equation*}
	\alpha = (1-\theta)\alpha_{0} + \theta\alpha_{1}, \qquad \frac{1}{p} = \frac{%
		1-\theta}{p_{0}} + \frac{\theta}{p_{1}}, \qquad \frac{1}{q} = \frac{1-\theta%
	}{q_{0}} + \frac{\theta}{q_{1}}.
	\end{equation*}
	Then the interpolation inequalities 
	\begin{equation}
	\|f\|_{\dot{K}_{p}^{\alpha,q}(\Omega)} \leq \|f\|_{\dot{K}%
		_{p_{0}}^{\alpha_{0},q_{0}}(\Omega)}^{1-\theta} \|f\|_{\dot{K}%
		_{p_{1}}^{\alpha_{1},q_{1}}(\Omega)}^{\theta}  \label{Interpolation}
	\end{equation}
	hold for all $f\in \dot{K}_{p_{0}}^{\alpha_{0},q_{0}}(\Omega) \cap \dot{K}%
	_{p_{1}}^{\alpha_{1},q_{1}}(\Omega)$, and 
	\begin{equation}
	\|f\|_{K_{p}^{\alpha,q}(\Omega)} \leq
	\|f\|_{K_{p_{0}}^{\alpha_{0},q_{0}}(\Omega)}^{1-\theta}
	\|f\|_{K_{p_{1}}^{\alpha_{1},q_{1}}(\Omega)}^{\theta}  \label{Interpolation1}
	\end{equation}
	hold for all $f\in K_{p_{0}}^{\alpha_{0},q_{0}}(\Omega) \cap
	K_{p_{1}}^{\alpha_{1},q_{1}}(\Omega)$.
\end{lemma}

The aim of the remainder of this section is to present equivalent norms for
Herz spaces; see \cite{Fle74}.

\begin{theorem}
	\label{equivalent1} Let $0<p,q\leq \infty$ and $\alpha\in \mathbb{R}$. Then 
	\begin{equation*}
	\|f\|_{\dot{K}_{p}^{\alpha,q}}^{\bullet} = \left( \int_{0}^{\infty}
	t^{\alpha q} \left( \int_{R(2t)}|f(x)|^{p}\,dx \right)^{q/p} \frac{dt%
	}{t} \right)^{1/q}
	\end{equation*}
	defines an equivalent quasi-norm on $\dot{K}_{p}^{\alpha,q}$, with the usual
	modifications when $p=\infty$ and/or $q=\infty$.
	
	Moreover, 
	\begin{equation*}
	\|f\|_{K_{p}^{\alpha,q}}^{\bullet} = \|f\chi_{B_{0}}\|_{p} + \left(
	\int_{1}^{\infty} t^{\alpha q} \left( \int_{R(2t)}|f(x)|^{p}\,dx
	\right)^{q/p} \frac{dt}{t} \right)^{1/q}
	\end{equation*}
	defines an equivalent quasi-norm on $K_{p}^{\alpha,q}$, again with the usual
	modifications when $p=\infty$ and/or $q=\infty$.
\end{theorem}

This theorem plays an essential role in establishing the optimality of some
assumptions in the embeddings between Herz--Sobolev spaces and $%
L^{\infty}(\Omega)$.

\section{Herz-type Sobolev spaces}

In this section, we establish the basic properties of Herz-type Sobolev
spaces in analogy with the classical Sobolev spaces.

\begin{definition}
	Let $\Omega \subset \mathbb{R}^{n}$ be open, $\alpha \in \mathbb{R}$, $1\leq
	p,q\leq \infty$, and $m\in \mathbb{N}_{0}$.
	
	\medskip
	
	\noindent \textnormal{(i)} Let $(\alpha,p,q)\in V_{\alpha,p,q}$. The
	homogeneous Herz-type Sobolev space $\dot{K}_{p,m}^{\alpha,q}(\Omega)$ is
	defined as the set of all functions $f\in \dot{K}_{p}^{\alpha,q}(\Omega)$
	whose weak derivatives $D^{\beta}f$ belong to $\dot{K}_{p}^{\alpha,q}(%
	\Omega) $ for all multi-indices $\beta$ satisfying $|\beta|\leq m$.
	
	If $1\leq p,q<\infty$, we define the norm by 
	\begin{equation*}
	\|f\|_{\dot{K}_{p,m}^{\alpha,q}(\Omega)} = \left( \sum_{k=-\infty}^{\infty}
	2^{k\alpha q} \left( \sum_{|\beta|\leq m} \|(D^{\beta}f)\chi_{R_{k}\cap
		\Omega}\|_{p}^{p} \right)^{q/p} \right)^{1/q},
	\end{equation*}
	while for $q=\infty$ we set 
	\begin{equation*}
	\|f\|_{\dot{K}_{p,m}^{\alpha,\infty}(\Omega)} = \sup_{k\in \mathbb{Z}}
	2^{k\alpha} \left( \sum_{|\beta|\leq m} \|(D^{\beta}f)\chi_{R_{k}\cap
		\Omega}\|_{p}^{p} \right)^{1/p}.
	\end{equation*}
	
	\medskip
	
	\noindent \textnormal{(ii)} The nonhomogeneous Herz-type Sobolev space $%
	K_{p,m}^{\alpha,q}(\Omega)$ is defined as the set of all functions $f\in
	K_{p}^{\alpha,q}(\Omega)$ whose weak derivatives $D^{\beta}f$ belong to $%
	K_{p}^{\alpha,q}(\Omega)$ for all multi-indices $\beta$ satisfying $%
	|\beta|\leq m$.
	
	If $1\leq p,q<\infty$, we define the norm by 
	\begin{equation*}
	\|f\|_{K_{p,m}^{\alpha,q}(\Omega)} = \left( \sum_{k=0}^{\infty} 2^{k\alpha
		q} \left( \sum_{|\beta|\leq m} \|(D^{\beta}f)\chi_{\widehat{R}_{k}\cap
		\Omega}\|_{p}^{p} \right)^{q/p} \right)^{1/q},
	\end{equation*}
	while for $q=\infty$ we set 
	\begin{equation*}
	\|f\|_{K_{p,m}^{\alpha,\infty}(\Omega)} = \sup_{k\in \mathbb{N}_{0}}
	2^{k\alpha} \left( \sum_{|\beta|\leq m} \|(D^{\beta}f)\chi_{\widehat{R}%
		_{k}\cap \Omega}\|_{p}^{p} \right)^{1/p}.
	\end{equation*}
\end{definition}

\begin{remark}
	It is immediate that if $p=q$ and $\alpha=0$, then 
	\begin{equation*}
	\dot{K}_{p,m}^{0,p}(\Omega) = K_{p,m}^{0,p}(\Omega) = W_{p}^{m}(\Omega).
	\end{equation*}
\end{remark}

Let $J\in \mathcal{D}(\mathbb{R}^{n})$ be a nonnegative real-valued function
such that 
\begin{equation*}
J(x)\geq 0\quad \text{for all }x\in \mathbb{R}^{n},\qquad \text{\textrm{supp }%
}J\subset \overline{B(0,1)},\qquad \int_{\mathbb{R}^{n}}J(x)\,dx=1.
\end{equation*}%
For $\varepsilon >0$, we define 
\begin{equation*}
J_{\varepsilon }(x)=\varepsilon ^{-n}J\left( \frac{x}{\varepsilon }\right)
,\qquad x\in \mathbb{R}^{n}.
\end{equation*}

\begin{theorem}
	\label{approximation1} Let $\Omega \subset \mathbb{R}^{n}$ be open. Let $%
	1\leq p,q<\infty$ and 
	\begin{equation*}
	-\frac{n}{p}<\alpha<n-\frac{n}{p},
	\end{equation*}
	or 
	\begin{equation*}
	1\leq p<\infty, \qquad \alpha=n-\frac{n}{p}, \qquad q=1.
	\end{equation*}
	Assume that $f\in \dot{K}_{p}^{\alpha,q}(\Omega)$ is extended by zero
	outside $\Omega$. Then 
	\begin{equation}
	\lim_{\varepsilon\to 0^{+}} \|J_{\varepsilon}*f-f\|_{\dot{K}%
		_{p}^{\alpha,q}(\Omega)} = 0.  \label{whatweneed}
	\end{equation}
\end{theorem}

\begin{proof}
	By \cite{Dr-Sobolev}, we have 
	\begin{equation}
	\lim_{\varepsilon\to 0^{+}} \|J_{\varepsilon}*\varphi-\varphi\|_{\dot{K}%
		_{p}^{\alpha,q}(\Omega)} = 0  \label{limit1}
	\end{equation}
	for every $\varphi\in C_{c}(\Omega)$ and every $\alpha\in \mathbb{R}$. By
	Theorem \ref{dense}, for any sufficiently small $\eta>0$, we may choose $%
	\varphi\in C_{c}(\Omega)$ such that 
	\begin{equation}
	\|f-\varphi\|_{\dot{K}_{p}^{\alpha,q}(\Omega)} \leq \frac{\eta}{3}.
	\label{limit2}
	\end{equation}
	
	\medskip
	
	\noindent \textbf{Step 1.} We prove \eqref{whatweneed}. We claim that 
	\begin{equation*}
	\|J_{\varepsilon}*f-J_{\varepsilon}*\varphi\|_{\dot{K}_{p}^{\alpha,q}(%
		\Omega)} \lesssim \|f-\varphi\|_{\dot{K}_{p}^{\alpha,q}(\Omega)},
	\end{equation*}
	where the implicit constant is independent of $\varepsilon$.
	
	Using \eqref{limit2}, for any sufficiently small $\eta_{1}>0$ we obtain 
	\begin{equation*}
	\|J_{\varepsilon}*f-J_{\varepsilon}*\varphi\|_{\dot{K}_{p}^{\alpha,q}(%
		\Omega)} \leq c \|f-\varphi\|_{\dot{K}_{p}^{\alpha,q}(\Omega)} \leq
	c\eta_{1}.
	\end{equation*}
	Choosing $\eta_{1}$ such that $c\eta_{1}<\eta/3$, and using \eqref{limit1},
	we find that 
	\begin{equation*}
	\|J_{\varepsilon}*\varphi-\varphi\|_{\dot{K}_{p}^{\alpha,q}(\Omega)} \leq 
	\frac{\eta}{3}
	\end{equation*}
	for sufficiently small $\varepsilon>0$. Hence, \eqref{whatweneed} follows
	from the triangle inequality.
	
	\medskip
	
	\noindent \textbf{Step 2.} We now prove the claim. Let $\varepsilon>0$ be
	sufficiently small. Write 
	\begin{equation*}
	\|J_{\varepsilon}*f-J_{\varepsilon}*\varphi\|_{\dot{K}_{p}^{\alpha,q}(%
		\Omega)}^{q} = I_{1,\varepsilon}+I_{2,\varepsilon},
	\end{equation*}
	where 
	\begin{equation*}
	I_{1,\varepsilon} = \sum_{\substack{ k\in \mathbb{Z}  \\ 2^{k}\leq
			4\varepsilon}} 2^{k\alpha q}
	\|(J_{\varepsilon}*(f-\varphi))\chi_{R_{k}\cap\Omega}\|_{p}^{q},
	\end{equation*}
	and 
	\begin{equation*}
	I_{2,\varepsilon} = \sum_{\substack{ k\in \mathbb{Z}  \\ 2^{k}>4\varepsilon 
	}} 2^{k\alpha q} \|(J_{\varepsilon}*(f-\varphi))\chi_{ R_{k}\cap\Omega}\|_{p}^{q}.
	\end{equation*}
	Assume first that 
	\begin{equation*}
	-\frac{n}{p}<\alpha<n-\frac{n}{p}, \qquad 1\leq p,q<\infty.
	\end{equation*}
	By H\"{o}lder's inequality, for every $x\in R_{k}\cap\Omega$ with $2^{k}\leq
	4\varepsilon$, we obtain 
	\begin{align*}
	|J_{\varepsilon}*(f-\varphi)(x)| &\leq \|J\|_{\infty} \varepsilon^{-n}
	\int_{B(0,6\varepsilon)\cap\Omega} |f(y)-\varphi(y)|\,dy \\
	&\lesssim \varepsilon^{-n} \sum_{j=-\infty}^{v} \int_{R_{j}\cap\Omega}
	|f(y)-\varphi(y)|\,dy \\
	&\lesssim \varepsilon^{-n} \sum_{j=-\infty}^{v} 2^{j\left(n-\frac{n}{p}%
		-\alpha\right)} 2^{j\alpha} \|(f-\varphi)\chi_{R_{j}\cap\Omega}\|_{p} \\
	&\lesssim 2^{-v\left(\frac{n}{p}+\alpha\right)} \|f-\varphi\|_{\dot{K}%
		_{p}^{\alpha,\infty}(\Omega)},
	\end{align*}
	where $2^{v-1}\leq \varepsilon<2^{v}, v\in \mathbb{Z}$ and the implicit
	constant is independent of $\varepsilon$. Consequently, 
	\begin{align*}
	I_{1,\varepsilon} &\lesssim \|f-\varphi\|_{\dot{K}_{p}^{\alpha,q}(%
		\Omega)}^{q} \sum_{\substack{ k\in \mathbb{Z}  \\ 2^{k}\leq 4\varepsilon}}
	2^{(k-v)\left(\frac{n}{p}+\alpha\right)} \\
	&\lesssim \|f-\varphi\|_{\dot{K}_{p}^{\alpha,q}(\Omega)}^{q},
	\end{align*}
	since $\alpha>-\frac{n}{p}$. Now assume that 
	\begin{equation*}
	\alpha=n-\frac{n}{p}, \qquad 1\leq p<\infty, \qquad q=1.
	\end{equation*}
	Then 
	\begin{align*}
	|J_{\varepsilon}*(f-\varphi)(x)| &\lesssim \varepsilon^{-n}
	\sum_{j=-\infty}^{v} 2^{j(n-\frac{n}{p})} \|(f-\varphi)\chi_{{R_{j}\cap\Omega}}\|_{p} \\
	&\lesssim \varepsilon^{-n} \|f-\varphi\|_{\dot{K}_{p}^{n-\frac{n}{p},1}(\Omega)},
	\end{align*}
	which yields 
	\begin{align*}
	I_{1,\varepsilon} &\lesssim \|f-\varphi\|_{\dot{K}_{p}^{n-\frac{n}{p},1}(\Omega)}
	\sum_{\substack{ k\in \mathbb{Z}  \\ 2^{k}\leq 4\varepsilon}} 2^{(k-v)n} \\
	&\lesssim \|f-\varphi\|_{\dot{K}_{p}^{n-\frac{n}{p},1}(\Omega)},
	\end{align*}
	with constants independent of $\varepsilon$.
	
	We now estimate $I_{2,\varepsilon}$. Let 
	\begin{equation*}
	\widetilde{R}_{k} = \{x\in \mathbb{R}^{n}:2^{k-1}\leq |x|\leq 2^{k+1}\}.
	\end{equation*}
	For every $x\in R_{k}\cap\Omega$ with $2^{k}>4\varepsilon$, we have 
	\begin{align*}
	|J_{\varepsilon}*(f-\varphi)(x)| &\leq \int_{\widetilde{R}_{k}\cap\Omega}
	J_{\varepsilon}(x-y) |f(y)-\varphi(y)|\,dy \\
	&= J_{\varepsilon}* \big((f-\varphi)\chi_{\widetilde{R}_{k}\cap\Omega}\big)(x).
	\end{align*}
	Therefore, 
	\begin{align*}
	I_{2,\varepsilon} &\lesssim \sum_{\substack{ k\in \mathbb{Z}  \\ %
			2^{k}>4\varepsilon}} 2^{k\alpha q} \|(f-\varphi)\chi_{\widetilde{R}%
		_{k}\cap\Omega}\|_{p}^{q} \\
	&\lesssim \|f-\varphi\|_{\dot{K}_{p}^{\alpha,q}(\Omega)}^{q}.
	\end{align*}
	This completes the proof.
\end{proof}

\begin{remark}
	Although \eqref{limit1} was proved in \cite{Dr-Sobolev} only in the case $%
	0\in \Omega$, the same techniques can easily be adapted to arbitrary open
	sets $\Omega$.
\end{remark}

The following theorem can be proved by arguments similar to those used in
Theorem \ref{approximation1}.

\begin{theorem}
	\label{approximation1 copy(1)} Let $\Omega \subset \mathbb{R}^{n}$ be an
	open set. Let $1\leq p,q<\infty$ and $\alpha\in \mathbb{R}$. Assume that $%
	f\in K_{p}^{\alpha,q}(\Omega)$ is extended by zero outside $\Omega$. Then 
	\begin{equation*}
	\lim_{\varepsilon\to 0^{+}}
	\|J_{\varepsilon}*f-f\|_{K_{p}^{\alpha,q}(\Omega)} = 0.
	\end{equation*}
\end{theorem}

As in the classical Sobolev setting (see \cite{AdamsFournier03} and \cite%
{Burenkov}), we also have the following results.

\begin{lemma}
	\label{productsmooth} Let $\Omega \subset \mathbb{R}^{n}$ be open, $%
	\alpha\in \mathbb{R}$, $1\leq p,q\leq \infty$, and $m\in \mathbb{N}_{0}$.
	Then for every $\varphi\in \mathcal{D}(\Omega)$ there exists a constant $c>0 
	$ such that 
	\begin{equation*}
	\|\varphi f\|_{\dot{K}_{p,m}^{\alpha,q}(\Omega)} \leq c \|f\|_{\dot{K}%
		_{p,m}^{\alpha,q}(\Omega)}
	\end{equation*}
	for all $f\in \dot{K}_{p,m}^{\alpha,q}(\Omega)$, and 
	\begin{equation*}
	\|\varphi f\|_{K_{p,m}^{\alpha,q}(\Omega)} \leq c
	\|f\|_{K_{p,m}^{\alpha,q}(\Omega)}
	\end{equation*}
	for all $f\in K_{p,m}^{\alpha,q}(\Omega)$.
\end{lemma}

\begin{theorem}
	Let $\Omega \subset \mathbb{R}^{n}$ be open, $\alpha\in \mathbb{R}$, $1\leq
	p,q\leq \infty$, and $m\in \mathbb{N}_{0}$.
	
	\medskip
	
	\noindent \textnormal{(i)} Let $(\alpha,p,q)\in V_{\alpha,p,q}$. Then the
	homogeneous Herz-type Sobolev space $\dot{K}_{p,m}^{\alpha,q}(\Omega)$ is a
	Banach space.
	
	\medskip
	
	\noindent \textnormal{(ii)} The nonhomogeneous Herz-type Sobolev space $%
	K_{p,m}^{\alpha,q}(\Omega)$ is a Banach space.
\end{theorem}

Based on Theorems \ref{approximation1}, \ref{approximation1 copy(1)}, Lemma %
\ref{productsmooth}, and the arguments presented in \cite{AdamsFournier03}, we
obtain the following density result.

\begin{theorem}
	\label{density} Let $\Omega \subset \mathbb{R}^{n}$ be open, $\alpha\in 
	\mathbb{R}$, and $m\in \mathbb{N}_{0}$.
	
	\medskip
	
	\noindent \textnormal{(i)} Let $1\leq p,q<\infty$ and 
	\begin{equation*}
	-\frac{n}{p}<\alpha<n-\frac{n}{p},
	\end{equation*}
	or 
	\begin{equation*}
	1\leq p<\infty, \qquad \alpha=n-\frac{n}{p}, \qquad q=1.
	\end{equation*}
	Then  $C^{\infty}(\Omega)\cap \dot{K}_{p,m}^{\alpha,q}(\Omega)$ is dense in $\dot{K}_{p,m}^{\alpha,q}(\Omega)$.
	
	\medskip
	
	\noindent \textnormal{(ii)} Let $1\leq p,q<\infty$. Then 
	$C^{\infty}(\Omega)\cap K_{p,m}^{\alpha,q}(\Omega) $
	is dense in $K_{p,m}^{\alpha,q}(\Omega)$.
\end{theorem}

\section{Main tools}

The proofs of the main results rely on several fundamental techniques from the theory of function spaces and harmonic analysis.

\begin{definition}
	Let $v\in \mathbb{R}^{n}\setminus \{0\}$ and, for each $x\neq 0$, let $%
	\angle (x,v)$ denote the angle between the position vector $x$ and $v$. Let $%
	\kappa$ satisfy $0<\kappa <\pi$. The set 
	\begin{equation*}
	C=\{x\in \mathbb{R}^{n}:x=0 \text{ or } 0<|x|\leq \delta,\ \angle (x,v)\leq
	\kappa /2\}
	\end{equation*}
	is called a finite cone with vertex at the origin, height $\delta$, axis
	direction $v$, and aperture angle $\kappa$.
\end{definition}

\begin{remark}
	Let $C$ be a finite cone with vertex at the origin. Observe that 
	\begin{equation*}
	x+C=\{x+y:y\in C\}
	\end{equation*}
	is a finite cone with vertex at $x$, having the same dimensions and axis
	direction as $C$, obtained by a parallel translation of $C$.
\end{remark}

We are now in a position to define domains satisfying the cone condition.

\begin{definition}
	\label{conecondition} Let $\Omega \subset \mathbb{R}^{n}$ be open. We say
	that $\Omega$ satisfies the cone condition if there exists a finite cone $C$
	such that, for every $x\in \Omega$, there exists a finite cone $C_{x}\subset
	\Omega$ with vertex at $x$ that is congruent to $C$.
\end{definition}

\begin{remark}
	In Definition \ref{conecondition}, the cone $C_{x}$ is not necessarily
	obtained from $C$ by parallel translation, but rather by a rigid motion.
\end{remark}

The following statement, which can be found in \cite[Lemma 4.15]%
{AdamsFournier03}, plays an essential role in what follows.

\begin{lemma}
	\label{Sobolev-rep1} Let $\Omega \subset \mathbb{R}^{n}$ be a domain
	satisfying the cone condition. Then there exists a positive constant $K$,
	depending only on $m$, $n$, and the parameters $\delta$ and $\kappa$ of the
	cone $C$ appearing in the cone condition for $\Omega$, such that for every $%
	f\in C^{\infty}(\Omega)$, every $x\in \Omega$, and every $\tau$ satisfying $%
	0<\tau \leq \delta$, we have 
	\begin{equation*}
	|f(x)|\leq K\Bigg( \sum_{|\beta |\leq m-1}\tau ^{|\beta |-n}
	\int_{C_{x,\tau}}|D^{\beta }f(y)|\,dy + \sum_{|\beta |=m}\int_{C_{x,\tau }} 
	\frac{|D^{\beta }f(y)|}{|x-y|^{n-m}}\,dy \Bigg),
	\end{equation*}
	where 
	\begin{equation*}
	C_{x,\tau }=\{y\in C_{x}: y\in \overline{B(x,\tau )}\}.
	\end{equation*}
\end{lemma}

Let $0<\lambda <n$. The Riesz potential operator $\mathcal{I}_{\lambda }$ is
defined by 
\begin{equation*}
\mathcal{I}_{\lambda }f(x)=\int_{{\mathbb{R}^{n}}}\frac{f(y)}{%
	|x-y|^{n-\lambda }}dy.
\end{equation*}%
Let $p^{\ast }$ denote the Sobolev exponent determined by $\frac{1}{p^{\ast }%
}=\frac{1}{p}-\frac{\lambda }{n}.$ The following statement plays a crucial
role in our embedding results; see \cite[Theorem 2.3]{LiYang96}.

\begin{theorem}
	\label{result3}Let $0<\lambda <n$, $0<q_{0}\leq q_{1}\leq \infty $ and $1<p<%
	\frac{n}{\lambda }$. If 
	\begin{equation*}
	\lambda -\frac{n}{p}<\alpha <n-\frac{n}{p},
	\end{equation*}%
	then $\mathcal{I}_{\lambda }$ is bounded from $\dot{K}_{{p}}^{{\alpha ,}%
		q_{0}}({\mathbb{R}^{n}})$ into $\dot{K}_{p^{\ast }}^{{\alpha ,}q_{1}}({%
		\mathbb{R}^{n}})$. Also $\mathcal{I}_{\lambda }$ is bounded from\ $K_{{p}}^{{%
			\alpha ,}q_{0}}({\mathbb{R}^{n}})$ into $K_{p^{\ast }}^{{\alpha ,}q_{1}}({%
		\mathbb{R}^{n}})$.
\end{theorem}

We shall also need the following lemma, which is essentially a consequence
of Hardy's inequality in the sequence space $\ell^{q}$ together with H\"{o}%
lder's inequality.

\begin{lemma}
	\label{key1} Let $\Omega \subset \mathbb{R}^{n}$ be a domain satisfying the
	cone condition, and let $m\in \mathbb{N}_{0}$. Assume that 
	$1\leq p<\infty,1\leq q\leq \infty, 1\leq r<\infty,$
	$\alpha _{2}\in \mathbb{R}$, $\alpha _{1}>-\frac{n}{q}$, and 
	\begin{equation*}
	\frac{n}{p}-\frac{n}{q}+\alpha _{2}-\alpha _{1}\leq m\leq n.
	\end{equation*}
	Let $f\in \dot{K}_{p,m}^{\alpha _{2},r}(\Omega )$, and let $\delta$ denote
	the dimension of the cone $C$. Define 
	\begin{equation}
	\digamma _{k,\beta }(x) = \int_{B_{x,\delta ,k}} \frac{|D^{\beta }f(y)|}{%
		|x-y|^{n-m}}\,dy, \qquad x\in \Omega, \quad |\beta |\leq m,  \label{integral}
	\end{equation}
	where 
	\begin{equation*}
	B_{x,\delta ,k} = C_{x}\cap \{y:2^{k+2}<|x-y|\leq \delta \}.
	\end{equation*}
	Then 
	\begin{equation}
	\sum_{\substack{ k\in \mathbb{Z}  \\ \delta >2^{k-2}}} 2^{k\alpha _{1}r} %
	\big\| \digamma _{k,\beta }\chi _{R_{k}\cap \Omega } \big\|_{q}^{r} \lesssim
	\delta ^{(m-\frac{n}{p}+\frac{n}{q}-\alpha _{2}+\alpha _{1})r} \big\|f\big\|%
	_{\dot{K}_{p,m}^{\alpha _{2},\nu }(\Omega )}^{r},  \label{sum1}
	\end{equation}
	where 
	\begin{equation*}
	\nu = 
	\begin{cases}
	\infty , & \text{if } m>\dfrac{n}{p}-\dfrac{n}{q}+\alpha _{2}-\alpha _{1}, \\%
	[1ex] 
	r, & \text{if } m=\dfrac{n}{p}-\dfrac{n}{q}+\alpha _{2}-\alpha _{1},%
	\end{cases}%
	\end{equation*}
	and the implicit constant is independent of $\delta$. Moreover, 
	\begin{equation*}
	\sup_{\substack{ k\in \mathbb{Z}  \\ \delta >2^{k-2}}} 2^{-k\frac{n}{q}} \big\| %
	\digamma _{k,\beta }\chi _{R_{k}\cap \Omega } \big\|_{q} \lesssim \delta ^{m-%
		\frac{n}{p}-\alpha _{2}} \big\|f\big\|_{\dot{K}_{p,m}^{\alpha _{2},s}(\Omega
		)},
	\end{equation*}
	where 
	\begin{equation*}
	s= 
	\begin{cases}
	\infty , & \text{if } m>\dfrac{n}{p}+\alpha _{2}, \\[1ex] 
	1, & \text{if } m=\dfrac{n}{p}+\alpha _{2},%
	\end{cases}%
	\end{equation*}
	and the implicit constant is independent of $\delta$.
\end{lemma}

\begin{proof}
	Let $x\in R_{k}$, where $k\in \mathbb{Z}$ and $\delta >2^{k-2}$. Choose $%
	j\in \mathbb{Z}$ such that 
	$2^{j-1}\leq 5\delta <2^{j}$, and fix this integer throughout the paper. By similarity, it suffices to
	prove \eqref{sum1}. Observe that 
	\begin{equation*}
	V_{k,\delta}=\{y\in \mathbb{R}^{n}:2^{k+1}\leq |y|<5\delta \} \subset
	\bigcup_{i=k+2}^{j}R_{i}.
	\end{equation*}
	Hence, using \eqref{integral}, we obtain 
	\begin{align*}
	\digamma _{k,\beta }(x) \leq& \int_{V_{k,\delta} } \frac{%
		|D^{\beta }f(y)|}{|x-y|^{n-m}} \chi _{\Omega }(y)\,dy \\
	\leq& \sum_{i=k+2}^{j} \int_{R_{i}\cap \Omega } \frac{|D^{\beta }f(y)|}{%
		|x-y|^{n-m}}\,dy \\
	\lesssim& \sum_{i=k+2}^{j} 2^{i(m-n)} \int_{R_{i}\cap \Omega }|D^{\beta
	}f(y)|\,dy,
	\end{align*}
	where the last estimate follows from the fact that $m\leq n$ and 
	\begin{equation*}
	|x-y| \geq |y|-|x| \geq 2^{i-2}, \qquad x\in R_{k}, \quad y\in R_{i}, \quad
	i\in \{k+2,\dots ,j\}.
	\end{equation*}
	The implicit constant is independent of $k$ and $\delta$. Applying
	H\"older's inequality yields 
	\begin{equation}
	\digamma _{k,\beta }(x) \lesssim \sum_{i=k+2}^{j} 2^{i\mu } 2^{i\alpha _{2}} %
	\big\| (D^{\beta }f)\chi _{R_{i}\cap \Omega } \big\|_{p},  \label{sum2}
	\end{equation}
	where 
	\begin{equation*}
	\mu =m-\frac{n}{p}-\alpha _{2}.
	\end{equation*}
	
	We now distinguish three cases.
	
	\medskip
	
	\noindent \textit{Case 1.} $\mu <0$. Taking the $L^{q}(R_{k}\cap \Omega )$%
	-norm on both sides of \eqref{sum2} and using 
	\begin{equation}
	\big\| \chi _{R_{k}\cap \Omega } \big\|_{q} \lesssim 2^{\frac{kn}{q}},
	\label{est-norm}
	\end{equation}
	we obtain 
	\begin{equation}
	2^{k\alpha _{1}} \big\| \digamma _{k,\beta }\chi _{R_{k}\cap \Omega } \big\|%
	_{q} \lesssim 2^{k(m-\frac{n}{p}+\frac{n}{q}-\alpha _{2}+\alpha _{1})}
	\sum_{i\geq k+2} 2^{(i-k)\mu } 2^{i\alpha _{2}} \big\| (D^{\beta }f)\chi
	_{R_{i}\cap \Omega } \big\|_{p}.  \label{J-2,4}
	\end{equation}
	
	First, suppose that 
	\begin{equation*}
	m=\frac{n}{p}-\frac{n}{q}+\alpha _{2}-\alpha _{1}.
	\end{equation*}
	Since $\mu<0$, Lemma \ref{lq-inequality} implies that the $\ell ^{r}$-norm
	of the right-hand side of \eqref{J-2,4} is bounded by 
	\begin{equation*}
	c\, \big\|f\big\|_{\dot{K}_{p,m}^{\alpha _{2},r}(\Omega )}^{r}.
	\end{equation*}
	
	Now assume that 
	\begin{equation*}
	m>\frac{n}{p}-\frac{n}{q}+\alpha _{2}-\alpha _{1}.
	\end{equation*}
	Then 
	\begin{align*}
	\sum_{i\geq k+2} 2^{(i-k)\mu } 2^{i\alpha _{2}} \big\| (D^{\beta }f)\chi
	_{R_{i}\cap \Omega } \big\|_{p} \lesssim& \big\|f\big\|_{\dot{K}%
		_{p,m}^{\alpha _{2},\infty }(\Omega )} \sum_{i\geq k+2}2^{(i-k)\mu } \\
	\lesssim& \big\|f\big\|_{\dot{K}_{p,m}^{\alpha _{2},\infty }(\Omega )},
	\end{align*}
	because $\mu<0$. Consequently, 
	\begin{align*}
	\sum_{\substack{ k\in \mathbb{Z}  \\ \delta >2^{k-2}}} 2^{k\alpha _{1}r} %
	\big\| \digamma _{k,\beta }\chi _{R_{k}\cap \Omega } \big\|_{q}^{r}
	\lesssim& \big\|f\big\|_{\dot{K}_{p,m}^{\alpha _{2},\infty }(\Omega )}^{r}
	\sum_{\substack{ k\in \mathbb{Z}  \\ \delta >2^{k-2}}} 2^{k(m-\frac{n}{p}+%
		\frac{n}{q}-\alpha _{2}+\alpha _{1})r} \\
	\lesssim& \delta ^{(m-\frac{n}{p}+\frac{n}{q}-\alpha _{2}+\alpha _{1})r} %
	\big\|f\big\|_{\dot{K}_{p,m}^{\alpha _{2},\infty }(\Omega )}^{r}.
	\end{align*}
	
	\medskip
	
	\noindent \textit{Case 2.} $\mu >0$. Since $\mu>0$, we have 
	\begin{align*}
	\digamma _{k,\beta }(x) \lesssim& \delta ^{\mu } \sum_{i=-\infty }^{j}
	2^{(i-j)\mu } 2^{i\alpha _{2}} \big\| (D^{\beta }f)\chi _{R_{i}\cap \Omega } %
	\big\|_{p} \\
	\lesssim& \delta ^{\mu } \sup_{i\in \mathbb{Z}} \left( 2^{i\alpha _{2}} %
	\big\| (D^{\beta }f)\chi _{R_{i}\cap \Omega } \big\|_{p} \right) \\
	\lesssim& \delta ^{\mu } \big\|f\big\|_{\dot{K}_{p,m}^{\alpha _{2},\infty
		}(\Omega )}.
	\end{align*}
	Therefore, by \eqref{est-norm}, 
	\begin{align*}
	\sum_{\substack{ k\in \mathbb{Z}  \\ \delta >2^{k-2}}} 2^{k\alpha _{1}r} %
	\big\| \digamma _{k,\beta }\chi _{R_{k}\cap \Omega } \big\|_{q}^{r}
	\lesssim& \delta ^{\mu r} \big\|f\big\|_{\dot{K}_{p,m}^{\alpha _{2},\infty
		}(\Omega )}^{r} \sum_{\substack{ k\in \mathbb{Z}  \\ \delta >2^{k-2}}}
	2^{kr(\alpha _{1}+\frac{n}{q})} \\
	\lesssim& \delta ^{(m-\frac{n}{p}+\frac{n}{q}-\alpha _{2}+\alpha _{1})r} %
	\big\|f\big\|_{\dot{K}_{p,m}^{\alpha _{2},\infty }(\Omega )}^{r},
	\end{align*}
	since $\alpha _{1}>-\frac{n}{q}$.
	
	\medskip
	
	\noindent \textit{Case 3.} $\mu =0$. In this case, 
	\begin{equation*}
	\digamma _{k,\beta }(x) \lesssim \sum_{i=k+2}^{j+1} 2^{i\alpha _{2}} \big\| %
	(D^{\beta }f)\chi _{R_{i}\cap \Omega } \big\|_{p} \lesssim (j-k) \big\|f%
	\big\|_{\dot{K}_{p,m}^{\alpha _{2},\infty }(\Omega )}.
	\end{equation*}
	Hence, 
	\begin{align*}
	\sum_{\substack{ k\in \mathbb{Z}  \\ \delta >2^{k-2}}} 2^{k\alpha _{1}r} %
	\big\| \digamma _{k,\beta }\chi _{R_{k}\cap \Omega } \big\|_{q}^{r}
	\lesssim& \big\|f\big\|_{\dot{K}_{p,m}^{\alpha _{2},\infty }(\Omega )}^{r}
	\sum_{\substack{ k\in \mathbb{Z}  \\ \delta >2^{k-2}}} (j-k) 2^{kr(\alpha
		_{1}+\frac{n}{q})} \\
	\lesssim& \delta ^{(\frac{n}{q}+\alpha _{1})r} \big\|f\big\|_{\dot{K}%
		_{p,m}^{\alpha _{2},\infty }(\Omega )}^{r},
	\end{align*}
	where we again used \eqref{est-norm} and the condition $\alpha _{1}>-\frac{n%
	}{q}$.
	
	This completes the proof.
\end{proof}

Next, we present the nonhomogeneous version of Lemma \ref{key1}.

\begin{lemma}
	\label{key1 copy(1)} Under the assumptions of Lemma \ref{key1}, we have 
	\begin{equation*}
	\sum_{\substack{ k\in \mathbb{N}  \\ \delta >2^{k-2}}} 2^{k\alpha _{1}r} %
	\big\| \digamma _{k,\beta }\chi _{R_{k}\cap \Omega } \big\|_{q}^{r} \lesssim
	\delta ^{(m-\frac{n}{p}+\frac{n}{q}-\alpha _{2}+\alpha _{1})r} \big\|f\big\|%
	_{K_{p,m}^{\alpha _{2},\nu }(\Omega )}^{r}
	\end{equation*}
	for every $f\in K_{p,m}^{\alpha _{2},\nu }(\Omega )$. Moreover, 
	\begin{equation*}
	\sup_{\substack{ k\in \mathbb{N}  \\ \delta >2^{k-2}}} 2^{-k\frac{n}{q}} \big\| %
	\digamma _{k,\beta }\chi _{R_{k}\cap \Omega } \big\|_{q} \lesssim \delta ^{m-%
		\frac{n}{p}-\alpha _{2}} \big\|f\big\|_{K_{p,m}^{\alpha _{2},s}(\Omega )}
	\end{equation*}
	for every $f\in K_{p,m}^{\alpha _{2},s}(\Omega )$.
\end{lemma}

\begin{remark}
	Let $\Omega \subset \mathbb{R}^{n}$ be a domain satisfying the cone
	condition. Let 
	$
	1\leq p<\infty, 1\leq q\leq \infty,  \alpha \in \mathbb{R}.
	$
	Assume that 
	\begin{equation*}
	\frac{n}{p}+\alpha \in \mathbb{N}_{0}, \qquad \frac{n}{p}+\alpha \leq n.
	\end{equation*}
	In view of the proof of Lemma \ref{key1}, we readily obtain 
	\begin{equation*}
	\sup_{\substack{ k\in \mathbb{N}  \\ \delta >2^{k-2}}} 2^{-k \frac{n}{q}} \big\| %
	\digamma _{k,\beta }\chi _{R_{k}\cap \Omega } \big\|_{q} \lesssim (\log
	_{2}(5\delta )) \big\|f\big\|_{K_{p,\frac{n}{p}+\alpha }^{\alpha
			,\infty }(\Omega )}
	\end{equation*}
	for every $f\in K_{p,\frac{n}{p}+\alpha }^{\alpha ,\infty }(\Omega )$,
	where the implicit constant is independent of $\delta$.
\end{remark}

The case $m>n$ in Lemma \ref{key1} is treated in the following lemma.

\begin{lemma}
	\label{key1 copy(4)} Let $\Omega \subset \mathbb{R}^{n}$ be a domain
	satisfying the cone condition, and let $m\in \mathbb{N}_{0}$. Assume that 
	$
	1\leq p<\infty,  1\leq q\leq \infty,  \alpha _{2}\leq n-\frac{n}{%
		p}, \alpha _{1}>-\frac{n}{q},
	$
	and let $m>n$. Let $f\in \dot{K}_{p,m}^{\alpha _{2},\infty }(\Omega )$, and
	let $\delta$ denote the dimension of the cone $C$. Let $\digamma _{k,\beta }$
	be defined as in \eqref{integral}. Then 
	\begin{equation}
	\sum_{\substack{ k\in \mathbb{Z}  \\ \delta >2^{k-2}}} 2^{k\alpha _{1}r} %
	\big\| \digamma _{k,\beta }\chi _{R_{k}\cap \Omega } \big\|_{q}^{r} \lesssim
	\delta ^{(m-\frac{n}{p}+\frac{n}{q}-\alpha _{2}+\alpha _{1})r} \big\|f\big\|%
	_{\dot{K}_{p,m}^{\alpha _{2},\infty }(\Omega )}^{r}.  \label{key1-est}
	\end{equation}
	In addition, 
	\begin{equation*}
	\sup_{\substack{ k\in \mathbb{Z}  \\ \delta >2^{k-2}}} 2^{-k\frac{n}{q}} \big\| %
	\digamma _{k,\beta }\chi _{R_{k}\cap \Omega } \big\|_{q} \lesssim \delta ^{m-%
		\frac{n}{p}-\alpha _{2}} \big\|f\big\|_{\dot{K}_{p,m}^{\alpha _{2},s}(\Omega
		)},
	\end{equation*}
	where 
	\begin{equation*}
	s= 
	\begin{cases}
	\infty , & \text{if } \alpha _{2}<n-\dfrac{n}{p}, \\[1ex] 
	1, & \text{if } \alpha _{2}=n-\dfrac{n}{p}.%
	\end{cases}%
	\end{equation*}
\end{lemma}

\begin{proof}
	By similarity, it suffices to prove \eqref{key1-est}. Let $j\in \mathbb{Z}$
	be chosen as in Lemma \ref{key1}. Since $m>n$, we have 
	$
	|x-y|^{m-n}\leq 2^{(j+1)(m-n)},  x\in R_{k},  y\in R_{i}, 
	k+2\leq i\leq j.
	$
	Therefore, 
	\begin{equation*}
	\digamma _{k,\beta }(x) \lesssim \delta ^{m-n} \sum_{i=k+2}^{j}
	\int_{R_{i}\cap \Omega }|D^{\beta }f(y)|\,dy
	\end{equation*}
	for every $x\in R_{k}$ and every $k\in \mathbb{Z}$ satisfying $\delta
	>2^{k-2}$, where the implicit constant is independent of $x$ and $k$.
	
	Applying H\"older's inequality, we obtain 
	\begin{equation*}
	\digamma _{k,\beta }(x) \lesssim \delta ^{m-n} \sum_{i=k+2}^{j+1} 2^{i(n-%
		\frac{n}{p}-\alpha _{2})} 2^{i\alpha _{2}} \big\| (D^{\beta }f)\chi
	_{R_{i}\cap \Omega } \big\|_{p}.
	\end{equation*}
	
	We distinguish two cases.
	
	\medskip
	
	\noindent \textit{Case 1.} $\alpha _{2}<n-\frac{n}{p}$. Since $n-\frac{n}{p%
	}-\alpha _{2}>0$, we have 
	\begin{align}
	\digamma _{k,\beta }(x) \lesssim& \delta ^{m-\frac{n}{p}-\alpha _{2}}
	\sum_{i=k+2}^{j} 2^{(i-j)(n-\frac{n}{p}-\alpha _{2})} 2^{i\alpha _{2}} %
	\big\| (D^{\beta }f)\chi _{R_{i}\cap \Omega } \big\|_{p}  \notag \\
	\lesssim& \delta ^{m-\frac{n}{p}-\alpha _{2}} \big\|f\big\|_{\dot{K}%
		_{p,m}^{\alpha _{2},\infty }(\Omega )}.  \label{inequality}
	\end{align}
	Combining \eqref{inequality} with %
	\eqref{est-norm}, we obtain 
	\begin{align*}
	\sum_{\substack{ k\in \mathbb{Z}  \\ \delta >2^{k-2}}} 2^{k\alpha _{1}r} %
	\big\| \digamma _{k,\beta }\chi _{R_{k}\cap \Omega } \big\|_{q}^{r}
	\lesssim& \delta ^{(m-\frac{n}{p}-\alpha _{2})r} \big\|f\big\|_{\dot{K}%
		_{p,m}^{\alpha _{2},\infty }(\Omega )}^{r} \sum_{\substack{ k\in \mathbb{Z} 
			\\ \delta >2^{k-2}}} 2^{kr(\alpha _{1}+\frac{n}{q})} \\
	\lesssim& \delta ^{(m-\frac{n}{p}+\frac{n}{q}-\alpha _{2}+\alpha _{1})r} %
	\big\|f\big\|_{\dot{K}_{p,m}^{\alpha _{2},\infty }(\Omega )}^{r},
	\end{align*}
	since $\alpha _{1}>-\frac{n}{q}$. The implicit constant is independent of $%
	\delta$.
	
	\medskip
	
	\noindent \textit{Case 2.} $\alpha _{2}=n-\frac{n}{p}$. In this case, 
	\begin{equation*}
	\digamma _{k,\beta }(x) \lesssim \delta ^{m-n} \sum_{i=k+2}^{j+1} 2^{i\alpha
		_{2}} \big\| (D^{\beta }f)\chi _{R_{i}\cap \Omega } \big\|_{p} \lesssim
	\delta ^{m-n}(j-k) \big\|f\big\|_{\dot{K}_{p,m}^{\alpha _{2},\infty }(\Omega
		)}.
	\end{equation*}
	Consequently, 
	\begin{align*}
	\sum_{\substack{ k\in \mathbb{Z}  \\ \delta >2^{k-2}}} 2^{k\alpha _{1}r} %
	\big\| \digamma _{k,\beta }\chi _{R_{k}\cap \Omega } \big\|_{q}^{r}
	\lesssim& \delta ^{(m-n)r} \big\|f\big\|_{\dot{K}_{p,m}^{\alpha _{2},\infty
		}(\Omega )}^{r} \sum_{\substack{ k\in \mathbb{Z}  \\ \delta >2^{k-2}}} (j-k)
	2^{kr(\alpha _{1}+\frac{n}{q})} \\
	\lesssim& \delta ^{(m-n+\frac{n}{q}+\alpha _{1})r} \big\|f\big\|_{\dot{K}%
		_{p,m}^{\alpha _{2},\infty }(\Omega )}^{r},
	\end{align*}
	where \eqref{est-norm} has been used.
	
	This completes the proof.
\end{proof}

The nonhomogeneous version of Lemma \ref{key1 copy(4)} is stated below.

\begin{lemma}
	\label{key1 copy(2)} Under the assumptions of Lemma \ref{key1 copy(4)}, we
	have 
	\begin{equation*}
	\sum_{\substack{ k\in \mathbb{N}  \\ \delta >2^{k-2}}} 2^{k\alpha _{1}r} %
	\big\| \digamma _{k,\beta }\chi _{R_{k}\cap \Omega } \big\|_{q}^{r} \lesssim
	\delta ^{(m-\frac{n}{p}+\frac{n}{q}-\alpha _{2}+\alpha _{1})r} \big\|f\big\|%
	_{K_{p,m}^{\alpha _{2},\infty }(\Omega )}^{r}
	\end{equation*}
	for every $f\in K_{p,m}^{\alpha _{2},\infty }(\Omega )$.
	
	Moreover, 
	\begin{equation*}
	\sup_{\substack{ k\in \mathbb{N}  \\ \delta >2^{k-2}}} 2^{-k\frac{n}{q}} \big\| %
	\digamma _{k,\beta }\chi _{R_{k}\cap \Omega } \big\|_{q} \lesssim \delta ^{m-%
		\frac{n}{p}-\alpha _{2}} \big\|f\big\|_{K_{p,m}^{\alpha _{2},s}(\Omega )}
	\end{equation*}
	for every $f\in K_{p,m}^{\alpha _{2},s}(\Omega )$, where 
	\begin{equation*}
	s= 
	\begin{cases}
	\infty , & \text{if } \alpha _{2}<n-\dfrac{n}{p}, \\[1ex] 
	1, & \text{if } \alpha _{2}=n-\dfrac{n}{p}.%
	\end{cases}%
	\end{equation*}
\end{lemma}

\begin{remark}
	Let $\Omega \subset \mathbb{R}^{n}$ be a domain satisfying the cone
	condition. Let 
	$1\leq p<\infty,1\leq q\leq \infty,  m\in \mathbb{N}, 
	m>n$.
	In view of the proof of Lemma \ref{key1 copy(4)}, we readily obtain 
	\begin{equation*}
	\sup_{\substack{ k\in \mathbb{N}  \\ \delta >2^{k-2}}} 2^{-k\frac{n}{q}} \big\| %
	\digamma _{k,\beta }\chi _{R_{k}\cap \Omega } \big\|_{q} \lesssim (\log
	_{2}(5\delta )) \big\|f\big\|_{K_{p,m}^{n-\frac{n}{p},\infty }(\Omega
		)}
	\end{equation*}
	for every $f\in \dot{K}_{p,m}^{n-\frac{n}{p},\infty }(\Omega )$, where the
	implicit constant is independent of $\delta$.
\end{remark}
The following lemma is a Hardy-type inequality, which is essentially a
consequence of Young's inequality in the sequence space $\ell^{q}$.

\begin{lemma}
	\label{lq-inequality} \textit{\ Let $0<a<1$ and $0<q\leq \infty$. Let $%
		\{\varepsilon_{k}\}_{k\in \mathbb{Z}}$ be a sequence of nonnegative real
		numbers and define 
		\begin{equation*}
		\delta_{k} = \sum_{k\in \mathbb{Z}} a^{\,|j-k|}\varepsilon_{j}, \qquad k\in 
		\mathbb{Z}.
		\end{equation*}
		Then there exists a constant $c>0$, depending only on $a$ and $q$, such that 
		\begin{equation*}
		\|\{\delta_{k}\}_{k\in \mathbb{Z}}\|_{\ell^{q}} \leq c
		\|\{\varepsilon_{k}\}_{k\in \mathbb{Z}}\|_{\ell^{q}}.
		\end{equation*}
	}
\end{lemma}

\section{Sobolev embeddings}
In this section, after the preliminary preparations, we establish new embedding theorems for the function spaces under consideration, assuming suitable conditions on the parameters.

\begin{theorem}
	\label{embeddingsfirst} Let $\Omega \subset \mathbb{R}^{n}$ be a domain
	satisfying the cone condition, and let $m\in \mathbb{N}$. Assume that 
	$
	1<p<\infty,  1\leq r<\infty,
	$
	and 
	\begin{equation}
	\frac{n}{q} = \frac{n}{p}-m-\alpha _{1}+\alpha _{2} >0, \qquad -\frac{n}{q}%
	<\alpha _{1}\leq \alpha _{2}<n-\frac{n}{p}.  \label{sobolevassumption1}
	\end{equation}
	Assume furthermore that 
	$m>\alpha _{2}-\alpha _{1}.$
	Then 
	\begin{equation*}
	\dot{K}_{p,m}^{\alpha _{2},r}(\Omega ) \hookrightarrow \dot{K}_{q}^{\alpha
		_{1},r}(\Omega )
	\end{equation*}
	and 
	\begin{equation*}
	K_{p,m}^{\alpha _{2},r}(\Omega ) \hookrightarrow K_{q}^{\alpha
		_{1},r}(\Omega ).
	\end{equation*}
\end{theorem}

\begin{proof}
	By similarity, it suffices to consider the homogeneous spaces $\dot{K}%
	_{p,m}^{\alpha _{2},r}(\Omega )$. We shall use Theorem \ref{density}. Let 
	$f\in C^{\infty }(\Omega )\cap \dot{K}_{p,m}^{\alpha _{2},r}(\Omega ),$ and let $\delta$ denote the dimension of the cone $C$. We decompose $\|f\|_{%
		\dot{K}_{q}^{\alpha _{1},r}(\Omega )}^{r}$ as 
	\begin{equation*}
	\|f\|_{\dot{K}_{q}^{\alpha _{1},r}(\Omega )}^{r} = I_{1}+I_{2},
	\end{equation*}
	where 
	\begin{equation*}
	I_{1} = \sum_{\substack{ k\in \mathbb{Z}  \\ \delta \leq 2^{k+2}}}
	2^{k\alpha _{1}r} \big\| f\chi _{R_{k}\cap \Omega } \big\|_{q}^{r},\quad I_{2} = \sum_{\substack{ k\in \mathbb{Z}  \\ \delta >2^{k+2}}} 2^{k\alpha
		_{1}r} \big\| f\chi _{R_{k}\cap \Omega } \big\|_{q}^{r}.
	\end{equation*}
	
	We estimate $I_{1}$ and $I_{2}$ separately.
	
	\medskip
	
	\noindent \textbf{Estimate of $I_{1}$.} Let $x\in R_{k}\cap \Omega$, where $%
	k\in \mathbb{Z}$ and $\delta \leq 2^{k+2} $. By Lemma \ref{Sobolev-rep1},
	together with the fact that 
	$n-m-\alpha _{1}+\alpha _{2}>0$ and $ \alpha _{2}\geq \alpha _{1},$
	we obtain 
	\begin{align*}
	|f(x)| \lesssim& \sum_{|\beta |\leq m-1} \delta ^{|\beta |-n}
	\int_{C_{x,\delta }} |D^{\beta }f(y)|\,dy 
	+ \sum_{|\beta |=m} \int_{C_{x,\delta }} \frac{|D^{\beta }f(y)|}{|x-y|^{n-m}%
	}\,dy \\
	\lesssim& \sum_{|\beta |\leq m} \int_{C_{x,\delta }} \frac{|D^{\beta }f(y)|%
	} {|x-y|^{n-m-\alpha _{1}+\alpha _{2}}}\,dy.
	\end{align*}
	Hence, 
	\begin{equation*}
	|f(x)| \leq c \sum_{|\beta |\leq m} \mathcal{I}_{m-\alpha _{2}+\alpha _{1}} %
	\big( (D^{\beta }f)\chi _{\Omega } \big)(x),
	\end{equation*}
	where the constant $c>0$ is independent of $k$. Therefore, 
	\begin{align*}
	I_{1} \lesssim& \sum_{|\beta |\leq m} \sum_{\substack{ k\in \mathbb{Z}  \\ %
			\delta \leq 2^{k+2}}} 2^{k\alpha _{1}r} \big\| \mathcal{I}_{m-\alpha
		_{2}+\alpha _{1}} \big( (D^{\beta }f)\chi _{\Omega } \big) \big\|_{q}^{r} \\
	\lesssim& \sum_{|\beta |\leq m} \sum_{\substack{ k\in \mathbb{Z}  \\ \delta
			\leq 2^{k+2}}} 2^{k\alpha _{2}r} \big\| \mathcal{I}_{m-\alpha _{2}+\alpha
		_{1}} \big( (D^{\beta }f)\chi _{\Omega } \big) \big\|_{q}^{r},
	\end{align*}
	where we used the inequality 
	\begin{equation}
	\left( \sum_{|\beta |\leq m}a_{\beta } \right)^{r} \lesssim \sum_{|\beta
		|\leq m}a_{\beta }^{r}, \qquad a_{\beta }\geq 0,  \label{emb5}
	\end{equation}
	together with the fact that $2^{k\alpha _{1}} \lesssim 2^{k\alpha _{2}}$ for
	every $k\in \mathbb{Z}$ satisfying $\delta \leq 2^{k+2}$, since $\alpha
	_{2}\geq \alpha _{1}$. By Theorem \ref{result3}, there exists a constant $c>0
	$ such that 
	\begin{equation*}
	I_{1} \lesssim \sum_{|\beta |\leq m} \big\| \mathcal{I}_{m-\alpha
		_{2}+\alpha _{1}} \big( (D^{\beta }f)\chi _{\Omega } \big) \big\|_{\dot{K}%
		_{q}^{\alpha _{2},r}(\mathbb{R}^{n})}^{r} \leq c \|f\|_{\dot{K}%
		_{p,m}^{\alpha _{2},r}(\Omega )}^{r}.
	\end{equation*}
	
	\medskip
	
	\noindent \textbf{Estimate of $I_{2}$.} Let $k\in \mathbb{Z}$ be such that $%
	\delta >2^{k+2}$ and let   $x\in R_{k}\cap \Omega$. Define 
	\begin{equation*}
	J_{1,k}(x) = \sum_{|\beta |\leq m-1} \int_{C_{x,\delta }} |D^{\beta
	}f(y)|\,dy,
	\end{equation*}
	and 
	\begin{equation*}
	J_{2,k}(x) = \sum_{|\beta |=m} \int_{C_{x,\delta }} \frac{|D^{\beta }f(y)|}{%
		|x-y|^{n-m}}\,dy  .
	\end{equation*}
	
	\medskip
	
	\noindent \textbf{Estimate of $J_{1,k}$.} Since $\delta >2^{k+2}$, we have 
	\begin{equation*}
	J_{1,k}(x) \lesssim \sum_{|\beta |\leq m-1} \int_{B(0,5\delta )} |D^{\beta
	}f(y)| \chi _{\Omega }(y)\,dy.
	\end{equation*}
	Hence, 
	\begin{equation*}
	J_{1,k}(x) \lesssim \sum_{|\beta |\leq m-1} \sum_{i=-\infty }^{j}
	\int_{R_{i}\cap \Omega } |D^{\beta }f(y)|\,dy.
	\end{equation*}
	
	Applying H\"older's inequality yields 
	\begin{align*}
	J_{1,k}(x) \lesssim& \sum_{|\beta |\leq m-1} \sum_{i=-\infty }^{j} 2^{i(n-%
		\frac{n}{p})} \big\| (D^{\beta }f)\chi _{R_{i}\cap \Omega } \big\|_{p} \\
	\lesssim& \sum_{|\beta |\leq m-1} \sum_{i=-\infty }^{j} 2^{i(n-\frac{n}{p}%
		-\alpha _{2})} \sup_{i\in \mathbb{Z}} \left( 2^{i\alpha _{2}} \big\| %
	(D^{\beta }f)\chi _{R_{i}\cap \Omega } \big\|_{p} \right) \\
	\lesssim& \|f\|_{\dot{K}_{p,m}^{\alpha _{2},\infty }(\Omega )},
	\end{align*}
	where the last inequality follows from the assumption 
	$\alpha _{2}<n-\frac{n}{p}.$
	
	Consequently, 
	\begin{align*}
	\sum_{\substack{ k\in \mathbb{Z}  \\ \delta >2^{k+2}}} 2^{k\alpha _{1}r} %
	\big\| J_{1,k}\chi _{R_{k}\cap \Omega } \big\|_{q}^{r} \lesssim& \sum 
	_{\substack{ k\in \mathbb{Z}  \\ \delta >2^{k+2}}} 2^{k(\alpha _{1}+\frac{n}{%
			q})r} \|f\|_{\dot{K}_{p,m}^{\alpha _{2},r}(\Omega )}^{r} \\
	\lesssim& \|f\|_{\dot{K}_{p,m}^{\alpha _{2},r}(\Omega )}^{r},
	\end{align*}
	since $\alpha _{1}>-\frac{n}{q}$.
	
	\medskip
	
	\noindent \textbf{Estimate of $J_{2,k}$.} We will prove that 
	\begin{equation}
	\sum_{\substack{ k\in \mathbb{Z}  \\ \delta >2^{k+2}}} 2^{k\alpha _{1}r}
	\|J_{2,k}\|_{q}^{r} \lesssim \|f\|_{\dot{K}_{p,m}^{\alpha _{2},r}(\Omega
		)}^{r}.  \label{J2,k}
	\end{equation}
	
	We decompose $J_{2,k}$ as 
	\begin{align*}
	J_{2,k}(x) =& \sum_{|\beta |=m} \int_{C_{x,2^{k+2}}} \frac{|D^{\beta }f(y)|}{%
		|x-y|^{n-m}}\,dy + \sum_{|\beta |=m} \digamma _{k,\beta }(x) \\
	=:& J_{3,k}(x)+J_{4,k}(x).
	\end{align*}
	
	Since $\alpha _{2}\geq \alpha _{1}$ and $\delta >2^{k+2}$, we obtain 
	\begin{align*}
	J_{3,k}(x) \lesssim& 2^{k(\alpha _{2}-\alpha _{1})} \sum_{|\beta |=m}
	\int_{C_{x,2^{k+2}}} \frac{|D^{\beta }f(y)|} {|x-y|^{n-m-\alpha _{1}+\alpha
			_{2}}} \chi _{\Omega }(y)\,dy \\
	=& c 2^{k(\alpha _{2}-\alpha _{1})} \sum_{|\beta |=m} \mathcal{I}_{m-\alpha
		_{2}+\alpha _{1}} \big( (D^{\beta }f)\chi _{\Omega } \big)(x),
	\end{align*}
	where $c>0$ is independent of $k$. Using \eqref{emb5} together with Theorem %
	\ref{result3}, we deduce that 
	\begin{equation*}
	\sum_{\substack{ k\in \mathbb{Z}  \\ \delta >2^{k+2}}} 2^{k\alpha _{1}r} %
	\big\| J_{3,k}\chi _{R_{k}\cap \Omega } \big\|_{q}^{r} \leq c \|f\|_{\dot{K}%
		_{p,m}^{\alpha _{2},r}(\Omega )}^{r}.
	\end{equation*}
	
	Furthermore, Lemma \ref{key1} implies that 
	\begin{equation*}
	\sum_{\substack{ k\in \mathbb{Z}  \\ \delta >2^{k-2}}} 2^{k\alpha _{1}r} %
	\big\| J_{4,k}\chi _{R_{k}\cap \Omega } \big\|_{q}^{r} \lesssim \|f\|_{\dot{K%
		}_{p,m}^{\alpha _{2},r}(\Omega )}^{r}.
	\end{equation*}
	
	Combining the estimates for $J_{3,k}$ and $J_{4,k}$ yields \eqref{J2,k}.
	This completes the proof.
\end{proof}

\begin{remark}
	Let $\Omega \subset \mathbb{R}^{n}$ be a domain satisfying the cone
	condition. Theorem \ref{embeddingsfirst} includes the classical Sobolev
	inequality as a special case. Moreover, by combining Theorem \ref%
	{embeddingsfirst} with the embeddings \eqref{herz-emb}, we obtain 
	\begin{equation*}
	W_{p}^{m}(\Omega ,|\cdot |^{\alpha _{2}p}) \hookrightarrow \dot{K}%
	_{q}^{\alpha _{1},p}(\Omega ) \hookrightarrow L^{q}(\Omega ,|\cdot |^{\alpha
		_{1}q}),
	\end{equation*}
	under the assumptions of Theorem \ref{embeddingsfirst} with $r=p$. In
	particular, 
	\begin{equation*}
	W_{p}^{m}(\Omega ) \hookrightarrow \dot{K}_{q}^{0,p}(\Omega )
	\hookrightarrow L^{q}(\Omega )
	\end{equation*}
	whenever 
	$
	1<p<\infty,  0<m<\frac{n}{p},
	$
	and 
	\begin{equation*}
	\frac{n}{q}=\frac{n}{p}-m.
	\end{equation*}
\end{remark}

\begin{remark}
	Theorem \ref{embeddingsfirst} was established in \cite[Theorem 3.12]%
	{Dr-Sobolev} under the additional assumption 
	\begin{equation*}
	m-\frac{n}{p}<\alpha _{2}.
	\end{equation*}
\end{remark}

In the next theorem, we consider the case $p=q$ in Theorem \ref%
{embeddingsfirst}. .

\begin{theorem}
	\label{embedingsq=p} Let the domain $\Omega \subset \mathbb{R}^{n}$ satisfy
	the cone condition, and let $m\in \mathbb{N}$. Assume that $1\le p<\infty$, $%
	1\le r<\infty$, and 
	\begin{equation*}
	-\frac{n}{p}<\alpha_{1}\le \alpha_{2}<n-\frac{n}{p}, \qquad m\ge
	\alpha_{2}-\alpha_{1}.
	\end{equation*}
	Then 
	\begin{equation*}
	\dot{K}_{p,m}^{\alpha_{2},r}(\Omega) \hookrightarrow \dot{K}%
	_{p}^{\alpha_{1},r}(\Omega)
	\end{equation*}
	and 
	\begin{equation*}
	K_{p,m}^{\alpha_{2},r}(\Omega) \hookrightarrow K_{p}^{\alpha_{1},r}(\Omega).
	\end{equation*}
\end{theorem}

\begin{proof}
	By similarity, it suffices to consider the spaces $\dot{K}%
	_{p,m}^{\alpha_{2},r}(\Omega)$. We use Theorem \ref{density}. Let $f\in
	C^{\infty}(\Omega)\cap \dot{K}_{p,m}^{\alpha_{2},r}(\Omega)$, and let $%
	\delta $ denote the dimension of the cone $C$. Observe that 
	\begin{equation*}
	\|f\|_{\dot{K}_{p}^{\alpha_{1},r}(\Omega)}^{r} = \Lambda_{1}+\Lambda_{2},
	\end{equation*}
	where 
	\begin{equation*}
	\Lambda_{1} = \sum_{\substack{ k\in\mathbb{Z}  \\ \delta\le 2^{k-2}}}
	2^{k\alpha_{1}r} \|f\chi_{R_{k}\cap\Omega}\|_{p}^{r}, \qquad \Lambda_{2} =
	\sum_{\substack{ k\in\mathbb{Z}  \\ \delta>2^{k-2}}} 2^{k\alpha_{1}r}
	\|f\chi_{R_{k}\cap\Omega}\|_{p}^{r}.
	\end{equation*}
	
	Let $J_{1,k}$ and $J_{2,k}$ be defined as in Theorem \ref{embeddingsfirst}.
	
	\medskip
	
	\noindent \textbf{Estimate of $\Lambda_{1}$.} We distinguish two cases.
	
	\medskip
	
	\noindent \textit{Case 1.} $m>\alpha_{2}-\alpha_{1}$. Let $x\in
	R_{k}\cap\Omega$, $k\in\mathbb{Z}$, and assume that $\delta\le 2^{k-2}$.
	Then $C_{x,\delta}\subset \widetilde{R}_{k}$, where 
	\begin{equation*}
	\widetilde{R}_{k} = \{z:2^{k-2}\le |z|\le 2^{k+1}\}.
	\end{equation*}
	Hence 
	\begin{align*}
	J_{1,k}(x) &\le \sum_{|\beta|\le m-1} \int_{\overline{B(x,\delta)}}
	|D^{\beta}f(y)| \chi_{\widetilde{R}_{k}\cap\Omega}(y)\,dy \\
	&= \sum_{|\beta|\le m-1} \Big( |D^{\beta}f| \chi_{\widetilde{R}%
		_{k}\cap\Omega} \Big) * \chi_{\overline{B(0,\delta)}}(x).
	\end{align*}
	
	By Young's inequality, 
	\begin{align*}
	\Big\| |D^{\beta}f| \chi_{\widetilde{R}_{k}\cap\Omega} * \chi_{\overline{%
			B(0,\delta)}} \Big\|_{p} &\lesssim \|\chi_{\overline{B(0,\delta)}}\|_{1} \,
	\|(D^{\beta}f)\chi_{\widetilde{R}_{k}\cap\Omega}\|_{p} \\
	&\lesssim \|D^{\beta}f\chi_{\widetilde{R}_{k}\cap\Omega}\|_{p}
	\end{align*}
	for every $|\beta |\leq m-1$. Therefore, 
	\begin{align*}
	\sum_{\substack{ k\in\mathbb{Z}  \\ \delta\le 2^{k-2}}} 2^{k\alpha_{1}r}
	\|J_{1,k}\chi_{R_{k}\cap\Omega}\|_{p}^{r} &\lesssim\sum_{|\beta|\le m-1} \sum_{\substack{ k\in%
			\mathbb{Z}  \\ \delta\le 2^{k-2}}} 2^{k\alpha_{1}r} \|(D^{\beta}f)\chi_{%
		\widetilde{R}_{k}\cap\Omega}\|_{p}^{r} \\
	&\lesssim \sum_{|\beta|\le m-1}\sum_{\substack{ k\in\mathbb{Z}  \\ \delta\le 2^{k-2}}}
	2^{k\alpha_{2}r} \|(D^{\beta}f)\chi_{\widetilde{R}_{k}\cap\Omega}\|_{p}^{r},
	\end{align*}
	since $\alpha_{2}\ge \alpha_{1}$, where we used \eqref{emb5}.
	
	Next we
	estimate $J_{2,k}$. Observe that 
	\begin{align*}
	J_{2,k}(x) &\lesssim \sum_{|\beta|\le m} \int_{C_{x,\delta}} \frac{%
		|D^{\beta}f(y)|}{|x-y|^{n-m}} \chi_{\widetilde{R}_{k}}(y)\,dy \\
	&\lesssim 2^{k(\alpha_{2}-\alpha_{1})} \sum_{|\beta|\le m} \int_{\overline{%
			B(x,\delta)}} \frac{|D^{\beta}f(y)|} {|x-y|^{n-m+\alpha_{2}-\alpha_{1}}}
	\chi_{\widetilde{R}_{k}}(y)\,dy \\
	&= c\,2^{k(\alpha_{2}-\alpha_{1})} \sum_{|\beta|\le m} \Big( |D^{\beta}f|
	\chi_{\widetilde{R}_{k}\cap\Omega} \Big) * \Big( \chi_{\overline{B(0,\delta)}%
	} \omega_{m,\alpha_{1},\alpha_{2}} \Big)(x),
	\end{align*}
	because $\alpha_{2}\ge \alpha_{1}$ and $|x-y|\le 2^{k+2}$ for all $y\in%
	\widetilde{R}_{k}$ and $x\in R_{k}$, where 
	\begin{equation*}
	\omega_{m,\alpha_{1},\alpha_{2}}(x) = \frac{1}{|x|^{n-m+\alpha_{2}-%
			\alpha_{1}}}, \qquad x\neq 0,
	\end{equation*}
	and the positive constant $c$ is independent of $k$. Applying Young's
	inequality once again, we obtain  for every $|\beta |\leq m$
	\begin{align*}
	\Big\| |D^{\beta}f| \chi_{\widetilde{R}_{k}\cap\Omega} * \chi_{\overline{%
			B(0,\delta)}} \omega_{m,\alpha_{1},\alpha_{2}} \Big\|_{p} &\lesssim \| \chi_{%
		\overline{B(0,\delta)}} \omega_{m,\alpha_{1},\alpha_{2}} \|_{1} \,
	\|(D^{\beta}f)\chi_{\widetilde{R}_{k}\cap\Omega}\|_{p} \\
	&\lesssim \|(D^{\beta}f)\chi_{\widetilde{R}_{k}\cap\Omega}\|_{p},
	\end{align*}
	since $m>\alpha_{2}-\alpha_{1}$. Consequently, 
	\begin{align*}
	\Lambda_{1} &\lesssim \sum_{l=1}^{2} \sum_{\substack{ k\in\mathbb{Z}  \\ %
			\delta\le 2^{k-2}}} 2^{k\alpha_{1}r}
	\|J_{l,k}\chi_{R_{k}\cap\Omega}\|_{p}^{r} \\
	&\lesssim\sum_{|\beta|\le m} \sum_{k\in\mathbb{Z}} 2^{k\alpha_{2}r} \|(D^{\beta}f)\chi_{%
		\widetilde{R}_{k}\cap\Omega}\|_{p}^{r} \\
	&\lesssim \|f\|_{\dot{K}_{p,m}^{\alpha_{2},r}(\Omega)}^{r}.
	\end{align*}
	
	\medskip
	
	\noindent \textit{Case 2.} $m=\alpha_{2}-\alpha_{1}$. We only need to
	consider the case $\alpha_{2}>\alpha_{1}$ and estimate $J_{2,k}$. We have 
	\begin{align*}
	J_{2,k}(x) &\lesssim \sum_{|\beta|\le m} \int_{C_{x,\delta}} \frac{%
		|D^{\beta}f(y)|}{|x-y|^{n-m}} \chi_{\widetilde{R}_{k}}(y)\,dy \\
	&\le \sum_{|\beta|\le m} \int_{\overline{B(x,\delta)}} \frac{|D^{\beta}f(y)|%
	} {|x-y|^{n-\alpha_{2}+\alpha_{1}}} \chi_{\widetilde{R}_{k}}(y)\,dy \\
	&\le \sum_{|\beta|\le m} \Big( |D^{\beta}f| \chi_{\widetilde{R}%
		_{k}\cap\Omega} \Big) * \Big( \chi_{\overline{B(0,2^{k-2})}}
	\omega_{\alpha_{2}-\alpha_{1},0,0} \Big)(x).
	\end{align*}
	
	By Young's inequality, 
	\begin{align*}
	\Big\| |D^{\beta}f| \chi_{\widetilde{R}_{k}\cap\Omega} * \chi_{\overline{%
			B(0,2^{k-2})}} \omega_{\alpha_{2}-\alpha_{1},0,0} \Big\|_{p} &\lesssim \|
	\chi_{\overline{B(0,2^{k-2})}} \omega_{\alpha_{2}-\alpha_{1},0,0} \|_{1} \,
	\|D^{\beta}f\chi_{\widetilde{R}_{k}\cap\Omega}\|_{p} \\
	&\lesssim 2^{k(\alpha_{2}-\alpha_{1})} \|D^{\beta}f\chi_{\widetilde{R}%
		_{k}\cap\Omega}\|_{p}.
	\end{align*}
	for every $|\beta |\leq m$. Therefore, as in Case 1, 
	\begin{equation*}
	\Lambda_{1} \lesssim \|f\|_{\dot{K}_{p,m}^{\alpha_{2},r}(\Omega)}^{r}.
	\end{equation*}
	
	\medskip
	
	\noindent \textbf{Estimate of $\Lambda_{2}$.} Again we distinguish two cases.
	
	\medskip
	
	\noindent \textit{Case 1.} $m>\alpha_{2}-\alpha_{1}$. Let $x\in
	R_{k}\cap\Omega$, $k\in\mathbb{Z}$, and suppose that $\delta>2^{k-2}$. Let $%
	J_{1,k}$ and $J_{2,k}$ be defined as in Theorem \ref{embeddingsfirst}. Since 
	$\delta>2^{k-2}$, we have 
	\begin{align*}
	J_{1,k}(x) &\lesssim \sum_{|\beta|\le m-1} \int_{B(0,5\delta)}
	|D^{\beta}f(y)| \chi_{\Omega}(y)\,dy \\
	&\le \sum_{|\beta|\le m-1} \sum_{i=-\infty}^{j} \int_{\Omega}
	|D^{\beta}f(y)| \chi_{i}(y)\,dy .
	\end{align*}
	
	By H\"{o}lder's inequality, 
	\begin{align*}
	J_{1,k}(x) &\lesssim \sum_{|\beta|\le m-1} \sum_{i=-\infty}^{j} 2^{i(n-\frac{%
			n}{p})} \|(D^{\beta}f)\chi_{R_{i}\cap\Omega}\|_{p} \\
	&\lesssim \sum_{|\beta|\le m-1} \sum_{i=-\infty}^{j} 2^{i(n-\frac{n}{p}%
		-\alpha_{2})} \sup_{i\in\mathbb{Z}} \Big( 2^{i\alpha_{2}}
	\|(D^{\beta}f)\chi_{R_{i}\cap\Omega}\|_{p} \Big) \\
	&\lesssim \sum_{|\beta|\le m-1} \|D^{\beta}f\|_{\dot{K}_{p}^{\alpha_{2},%
			\infty}(\Omega)} \\
	&\lesssim \|f\|_{\dot{K}_{p,m}^{\alpha_{2},\infty}(\Omega)},
	\end{align*}
	where the last inequality follows from the assumption $\alpha_{2}<n-\frac{n}{%
		p}$.
	
	Hence 
	\begin{align*}
	\sum_{\substack{ k\in\mathbb{Z}  \\ \delta>2^{k-2}}} 2^{k\alpha_{1}r}
	\|J_{1,k}\chi_{R_{k}\cap\Omega}\|_{p}^{r} &\lesssim \sum_{\substack{ k\in%
			\mathbb{Z}  \\ \delta>2^{k-2}}} 2^{k(\alpha_{1}+\frac{n}{p})r} \|f\|_{\dot{K}%
		_{p,m}^{\alpha_{2},r}(\Omega)}^{r} \\
	&\lesssim \|f\|_{\dot{K}_{p,m}^{\alpha_{2},r}(\Omega)}^{r},
	\end{align*}
	because $\alpha_{1}>-\frac{n}{p}$.
	
	Next, we decompose $J_{2,k} =J_{3,k}+J_{4,k}$, where 
	\begin{equation*}
	J_{3,k}(x)=\sum_{|\beta|=m} \int_{C_{x,2^{k-2}}} \frac{|D^{\beta}f(y)|}{%
		|x-y|^{n-m}}\,dy ,\qquad J_{4,k}(x)= \sum_{|\beta|=m} \int_{B_{x,\delta,k-4}} \frac{%
		|D^{\beta}f(y)|}{|x-y|^{n-m}}\,dy .
	\end{equation*}
	
	Since $C_{x,2^{k-2}}\subset \widetilde{R}_{k}\cap\Omega$, we obtain 
	\begin{align*}
	J_{3,k}(x) &\lesssim \sum_{|\beta|=m} 2^{k(\alpha_{2}-\alpha_{1})}
	\int_{C_{x,2^{k-2}}} \frac{ |D^{\beta}f(y)| \chi_{\widetilde{R}%
			_{k}\cap\Omega}(y) } {|x-y|^{n-m-\alpha_{1}+\alpha_{2}}} \,dy \\
	&= c\,2^{k(\alpha_{2}-\alpha_{1})} \sum_{|\beta|=m} \Big( |D^{\beta}f| \chi_{%
		\widetilde{R}_{k}\cap\Omega} \Big) * \Big( \chi_{\overline{B(0,\delta)}}
	\omega_{m,\alpha_{1},\alpha_{2}} \Big)(x),
	\end{align*}
	because $\alpha_{2}\ge \alpha_{1}$ and $|x-y|\le 2^{k+2}$ for all $y\in 
	\widetilde{R}_{k}$ and $x\in R_{k}$. Young's inequality gives 
	\begin{align*}
	\Big\| |D^{\beta}f| \chi_{\widetilde{R}_{k}\cap\Omega} * \chi_{\overline{%
			B(0,\delta)}} \omega_{m,\alpha_{1},\alpha_{2}} \Big\|_{p} &\lesssim \| \chi_{%
		\overline{B(0,\delta)}} \omega_{m,\alpha_{1},\alpha_{2}} \|_{1} \,
	\|(D^{\beta}f)\chi_{\widetilde{R}_{k}\cap\Omega}\|_{p} \\
	&\lesssim \|(D^{\beta}f)\chi_{\widetilde{R}_{k}\cap\Omega}\|_{p},\qquad |\beta |= m,
	\end{align*}
	since $m>\alpha_{2}-\alpha_{1}$. Consequently, 
	\begin{equation*}
	\sum_{\substack{ k\in\mathbb{Z}  \\ \delta>2^{k-2}}} 2^{k\alpha_{1}r}
	\|J_{3,k}\chi_{R_{k}\cap\Omega}\|_{p}^{r} \lesssim \|f\|_{\dot{K}%
		_{p,m}^{\alpha_{2},r}(\Omega)}^{r},  \label{lambda22}
	\end{equation*}
	where we used \eqref{emb5}. We further decompose $J_{4,k}$ as 
	$J_{4,k} = J_{4,k}^{1}+J_{4,k}^{2},$
	where\begin{align*}
	J_{4,k}^{1} (x)=& \sum_{|\beta |=m} \int_{B_{x,2^{k+2},k-4}} \frac{|D^{\beta }f(y)|}{%
		|x-y|^{n-m}}\,dy,\qquad J_{4,k}^{2}(x)= \sum_{|\beta |=m} \digamma _{k,\beta }(x).
	\end{align*}
	Using H\"{o}lder's inequality, 
	\begin{align*}
	J_{4,k}^{1}(x) &\le 2^{k(m-n)} \sum_{|\beta|=m} \int_{|y|\le 2^{k+4}}
	|D^{\beta}f(y)| \chi_{\Omega}(y)\,dy \\
	&\lesssim 2^{k(m-n)} \sum_{|\beta|=m} \sum_{i=-\infty}^{k+4} 2^{i(n-\frac{n}{%
			p}-\alpha_{2})} 2^{i\alpha_{2}} \|(D^{\beta}f)\chi_{R_{i}\cap\Omega}\|_{p}.
	\end{align*}
	
	Since $\alpha_{2}<n-\frac{n}{p}$, 
	\begin{align*}
	2^{k\alpha_{1}} \|J_{4,k}^{1}\chi_{R_{k}\cap\Omega}\|_{p} &\lesssim
	2^{k(m-n+\alpha_{1}+\frac{n}{p})} \sum_{|\beta|=m} \sum_{i=-\infty}^{k+4}
	2^{i(n-\frac{n}{p}-\alpha_{2})} 2^{i\alpha_{2}}
	\|(D^{\beta}f)\chi_{R_{i}\cap\Omega}\|_{p} \\
	&\lesssim 2^{k(m+\alpha_{1}-\alpha_{2})}  \sum_{|\beta|=m}\sup_{i\in\mathbb{Z}} \Big( %
	2^{i\alpha_{2}} \|(D^{\beta}f)\chi_{R_{i}\cap\Omega}\|_{p} \Big) \\
	&\lesssim 2^{k(m+\alpha_{1}-\alpha_{2})} \|f\|_{\dot{K}_{p,m}^{%
			\alpha_{2},r}(\Omega)}.
	\end{align*}
	
	Therefore, 
	\begin{equation*}
	\sum_{\substack{ k\in\mathbb{Z}  \\ \delta>2^{k-2}}} 2^{k\alpha_{1}r}
	\|J_{4,k}^{1}\chi_{R_{k}\cap\Omega}\|_{p}^{r} \lesssim \|f\|_{\dot{K}%
		_{p,m}^{\alpha_{2},r}(\Omega)}^{r},
	\end{equation*}
	where we used the fact that $m>\alpha_{2}-\alpha_{1}$.
	
	Finally, by Lemmas \ref{key1} and \ref{key1 copy(4)}, 
	\begin{equation*}
	\sum_{\substack{ k\in\mathbb{Z}  \\ \delta>2^{k-2}}} 2^{k\alpha_{1}r}
	\|J_{4,k}^{2}\chi_{R_{k}\cap\Omega}\|_{p}^{r} \lesssim \|f\|_{\dot{K}%
		_{p,m}^{\alpha_{2},r}(\Omega)}^{r}.
	\end{equation*}
	
	\medskip
	
	\noindent \textit{Case 2.} $m=\alpha_{2}-\alpha_{1}$. It remains only to
	estimate $J_{2,k}$. We first observe that 
	\begin{align*}
	J_{3,k}(x) &\lesssim \sum_{|\beta|=m} \int_{C_{x,2^{k-2}}} \frac{
		|D^{\beta}f(y)| \chi_{\widetilde{R}_{k}\cap\Omega}(y) } {|x-y|^{n-%
			\alpha_{2}+\alpha_{1}}} \,dy \\
	&= c \sum_{|\beta|=m} \Big( |D^{\beta}f| \chi_{\widetilde{R}_{k}\cap\Omega} %
	\Big) * \Big( \chi_{\overline{B(0,2^{k-2})}}
	\omega_{\alpha_{2}-\alpha_{1},0,0} \Big)(x).
	\end{align*}
	
	We only need to consider the case $\alpha_{2}>\alpha_{1}$, Young's
	inequality yields 
	\begin{align*}
	\Big\| |D^{\beta}f| \chi_{\widetilde{R}_{k}\cap\Omega} * \chi_{\overline{%
			B(0,2^{k-2})}} \omega_{\alpha_{2}-\alpha_{1},0,0} \Big\|_{p} &\lesssim \|
	\chi_{\overline{B(0,2^{k-2})}} \omega_{\alpha_{2}-\alpha_{1},0,0} \|_{1} \,
	\|(D^{\beta}f)\chi_{\widetilde{R}_{k}\cap\Omega}\|_{p} \\
	&\lesssim 2^{k(\alpha_{2}-\alpha_{1})} \|(D^{\beta}f)\chi_{\widetilde{R}%
		_{k}\cap\Omega}\|_{p},\qquad |\beta|=m
	\end{align*}
	
	Hence, 
	\begin{equation*}
	\sum_{\substack{ k\in\mathbb{Z}  \\ \delta>2^{k-2}}} 2^{k\alpha_{1}r}
	\|J_{3,k}\chi_{R_{k}\cap\Omega}\|_{p}^{r} \lesssim \|f\|_{\dot{K}%
		_{p,m}^{\alpha_{2},r}(\Omega)}^{r},
	\end{equation*}
	where we used \eqref{emb5}. Next, 
	\begin{align*}
	J_{4,k}^{1}(x) &\le 2^{k(\alpha_{2}-\alpha_{1}-n)} \sum_{|\beta|=m}
	\int_{|y|\le 2^{k+4}} |D^{\beta}f(y)| \chi_{\Omega}(y)\,dy \\
	&\lesssim 2^{k(\alpha_{2}-\alpha_{1}-n)} \sum_{|\beta|=m}
	\sum_{i=-\infty}^{k+4} 2^{i(n-\frac{n}{p}-\alpha_{2})} 2^{i\alpha_{2}}
	\|(D^{\beta}f)\chi_{R_{i}\cap\Omega}\|_{p},
	\end{align*}
	by H\"{o}lder's inequality. Taking the $L^{p}(R_{k}\cap\Omega)$-norm on both
	sides yields 
	\begin{align*}
	2^{k\alpha_{1}} \|J_{4,k}^{1}\chi_{R_{k}\cap\Omega}\|_{p} &\lesssim
	2^{k(\alpha_{2}-n+\frac{n}{p})} \sum_{|\beta|=m} \sum_{i=-\infty}^{k+4}
	2^{i(n-\frac{n}{p}-\alpha_{2})} 2^{i\alpha_{2}}
	\|(D^{\beta}f)\chi_{R_{i}\cap\Omega}\|_{p} \\
	&= c \sum_{|\beta|=m} \sum_{i=-\infty}^{k+4} 2^{(i-k)(n-\frac{n}{p}%
		-\alpha_{2})} 2^{i\alpha_{2}} \|(D^{\beta}f)\chi_{R_{i}\cap\Omega}\|_{p}.
	\end{align*}
	
	Therefore, 
	\begin{equation*}
	\sum_{\substack{ k\in\mathbb{Z}  \\ \delta>2^{k-2}}} 2^{k\alpha_{1}r}
	\|J_{4,k}^{1}\chi_{R_{k}\cap\Omega}\|_{p}^{r} \lesssim \|f\|_{\dot{K}%
		_{p,m}^{\alpha_{2},r}(\Omega)}^{r},
	\end{equation*}
	by \eqref{emb5} and Lemma \ref{lq-inequality}. Finally, Lemmas \ref{key1}
	and \ref{key1 copy(4)} imply that 
	\begin{equation*}
	\sum_{\substack{ k\in\mathbb{Z}  \\ \delta>2^{k-2}}} 2^{k\alpha_{1}r}
	\|J_{4,k}^{2}\chi_{R_{k}\cap\Omega}\|_{p}^{r} \lesssim \|f\|_{\dot{K}%
		_{p,m}^{\alpha_{2},r}(\Omega)}^{r}.
	\end{equation*}
	
	This completes the proof.
\end{proof}

Next, we give an example showing that the assumption $m\ge
\alpha_{2}-\alpha_{1}$ in Theorem \ref{embedingsq=p} is essential.

\begin{example}
	\label{example1} Let $\Omega \subset \mathbb{R}^{n}$ be a domain satisfying
	the cone condition and such that $0\in \Omega$. Let $m\in \mathbb{N}$, $1\le
	r,v<\infty$, and 
	\begin{equation*}
	-\frac{n}{p}<\alpha_{1}\le \alpha_{2}<n-\frac{n}{p}.
	\end{equation*}
	Then $\dot{K}_{p,m}^{\alpha_{2},r}(\Omega)$ is not embedded into $\dot{K}%
	_{p}^{\alpha_{1},v}(\Omega)$ if 
	$m<\alpha_{2}-\alpha_{1}$.
	
	Indeed, let $u>0$ be such that 
	$\overline{B(0,2u)}\subset \Omega.
	$
	Let $\kappa<0$ be chosen later and define 
	\begin{equation*}
	f(x)=|x|^{\kappa}.
	\end{equation*}
	By induction on $|\beta|$, we have 
	\begin{equation*}
	D^{\beta}f(x)=P_{\beta}(x)|x|^{\kappa-2|\beta|},
	\end{equation*}
	where $P_{\beta}(x)$ is a homogeneous polynomial of degree $|\beta|$ in the
	components of $x$; see \cite[Example 4.41]{AdamsFournier03}. Hence 
	\begin{equation}
	|D^{\beta}f(x)|\le c|x|^{\kappa-|\beta|},  \label{prop1}
	\end{equation}
	where the positive constant $c$ is independent of $x$. Assume that $%
	|\beta|\le m$. Then 
	\begin{align*}
	\|D^{\beta}f\|_{\dot{K}_{p}^{\alpha_{2},r}(B(0,u))}^{r} &= \sum_{k\in\mathbb{%
			Z}} 2^{k\alpha_{2}r} \|(D^{\beta}f)\,\chi_{R_{k}\cap B(0,u)}\|_{p}^{r} \\
	&= \sum_{\substack{ k\in\mathbb{Z}  \\ 2^{k-1}\le u}} 2^{k\alpha_{2}r}
	\|(D^{\beta}f)\,\chi_{R_{k}\cap B(0,u)}\|_{p}^{r}.
	\end{align*}
	
	By \eqref{prop1}, 
	\begin{equation*}
	\int_{R_{k}\cap B(0,u)} |D^{\beta}f(x)|^{p}\,dx \le \int_{R_{k}}
	|x|^{(\kappa-|\beta|)p}\,dx \lesssim 2^{k(\kappa-|\beta|)p+n},
	\end{equation*}
	where the implicit constant is independent of $k$. Consequently, 
	\begin{equation*}
	\sum_{\substack{ k\in\mathbb{Z}  \\ 2^{k-1}\le u}} 2^{k(\alpha_{2}+\kappa-|%
		\beta|+\frac{n}{p})r} <\infty
	\end{equation*}
	provided that 
	\begin{equation*}
	\alpha_{2}+\kappa-|\beta|+\frac{n}{p}>0.
	\end{equation*}
	
	On the other hand, 
	\begin{equation*}
	\|f\|_{\dot{K}_{p}^{\alpha_{1},v}(B(0,u))}^{v} \ge \sum_{\substack{ k\in%
			\mathbb{Z}  \\ 2^{k}\le u}} 2^{k\alpha_{1}v} \|f\chi_{R_{k}\cap
		B(0,u)}\|_{p}^{v}.
	\end{equation*}
	Since $R_{k}\subset B(0,u)$ whenever $2^{k}\le u$, there exists a positive
	constant $c$ such that 
	\begin{equation*}
	\int_{R_{k}\cap B(0,u)} |f(x)|^{p}\,dx = \int_{R_{k}} |f(x)|^{p}\,dx \ge
	c\,2^{k\kappa p+n}.
	\end{equation*}
	Thus 
	$f\notin \dot{K}_{p}^{\alpha_{1},v}(B(0,u))$
	if 
	$\kappa<-\frac{n}{p}-\alpha_{1}.$
	
	Choose $\kappa$ so that 
	\begin{equation*}
	m-\frac{n}{p}-\alpha_{2} < \kappa < -\frac{n}{p}-\alpha_{1}.
	\end{equation*}
	Then 
	\begin{equation*}
	f\in \dot{K}_{p,m}^{\alpha_{2},r}(B(0,u)) \quad \text{but} \quad f\notin 
	\dot{K}_{p}^{\alpha_{1},v}(B(0,u)).
	\end{equation*}
	
	Let $\psi\in \mathcal{D}(\Omega)$ satisfy 
	$\text{\textrm{supp }%
	}\psi \subset \overline{B(0,2u)}$
	and 
	$\psi(x)=1 $ for $ x\in B(0,u).$
	Define 
	\begin{equation*}
	g(x)=\psi(x)f(x).
	\end{equation*}
	We claim that 
	\begin{equation*}
	f\in \dot{K}_{p,m}^{\alpha_{2},r}(B(0,u)) \quad \Longleftrightarrow \quad
	g\in \dot{K}_{p,m}^{\alpha_{2},r}(\Omega).
	\end{equation*}
	This proves the necessity of the condition 
	$m\ge
	\alpha_{2}-\alpha_{1}$.
	
	It remains to prove the claim. We only prove the direct implication. It is
	easy to see that 
	\begin{equation*}
	\int_{\Omega} |D^{\beta}g(x)|^{p}\chi_{k}(x)\,dx = \int_{B(0,u)}
	|D^{\beta}f(x)|^{p}\chi_{k}(x)\,dx + \int_{\widetilde{R}(u)}
	|D^{\beta}g(x)|^{p}\chi_{k}(x)\,dx,
	\end{equation*}
	where $|\beta|\le m$, 
	\begin{equation*}
	\widetilde{R}(u)=\{x\in\mathbb{R}^{n}:u\le |x|\le 2u\},
	\end{equation*}
	and $k\in\mathbb{Z}$. Observe that 
	\begin{equation}
	\|D^{\beta}g\|_{\dot{K}_{p}^{\alpha_{2},r}(\widetilde{R}(u))}^{r} = \sum 
	_{\substack{ k\in\mathbb{Z}  \\ u\approx 2^{k}}} 2^{k\alpha_{2}r}
	\|D^{\beta}g\,\chi_{\widetilde{R}(u)\cap R_{k}}\|_{p}^{r}.  \label{estg2}
	\end{equation}
	
	By Leibniz's formula, \eqref{prop1}, and the fact that $\psi\in \mathcal{D}%
	(\Omega)$, we obtain 
	\begin{equation}
	|D^{\beta}g(x)| \le \sum_{\gamma\le \beta} C_{\beta}^{\gamma}
	|D^{\beta-\gamma}\psi(x)| |D^{\gamma}f(x)| \lesssim |x|^{\kappa-|\beta|}
	\label{estg}
	\end{equation}
	for all $x\in \widetilde{R}(u)\cap R_{k}$, where the implicit constant is
	independent of $x$.
	
	Combining \eqref{estg2} and \eqref{estg}, we obtain 
	\begin{equation*}
	\|D^{\beta}g\|_{\dot{K}_{p}^{\alpha_{2},r}(\widetilde{R}(u))}^{r} \lesssim
	\sum_{\substack{ k\in\mathbb{Z}  \\ u\approx 2^{k}}} 2^{k(\alpha_{2}+%
		\kappa-|\beta|+\frac{n}{p})r} \lesssim 1.
	\end{equation*}
	This proves the claim. Similarly, one can show that 
	\begin{equation*}
	K_{p,m}^{\alpha_{2},r}(\Omega) \not\hookrightarrow
	K_{p}^{\alpha_{1},v}(\Omega)
	\end{equation*}
	if $m<\alpha_{2}-\alpha_{1}$.
\end{example}

Collecting the results of Theorems \ref{embedingsq=p} and \ref%
{embeddingsfirst}, together with the interpolation inequalities %
\eqref{Interpolation} and \eqref{Interpolation1}, we obtain the following
statement.

\begin{theorem}
	\label{embeddingsfirst copy(1)} Let $\Omega \subset \mathbb{R}^{n}$ be a
	domain satisfying the cone condition, and let $m\in \mathbb{N}_{0}$. Assume
	that $1\le r<\infty, 1<p<q<p^{\ast}<\infty, m>\alpha_{2}-\alpha_{1},$
	and 
	\begin{equation*}
	\frac{n}{p^{\ast}} = \frac{n}{p}-m+\alpha_{2}-\alpha_{1} >0,
	\end{equation*}
	with 
	\begin{equation*}
	-\frac{n}{q}<\alpha_{1}\le \alpha_{2}<n-\frac{n}{p}.
	\end{equation*}
	Then 
	\begin{equation*}
	\dot{K}_{p,m}^{\alpha_{2},r}(\Omega) \hookrightarrow \dot{K}%
	_{q}^{\alpha_{1},r}(\Omega)
	\end{equation*}
	and 
	\begin{equation}
	K_{p,m}^{\alpha_{2},r}(\Omega) \hookrightarrow K_{q}^{\alpha_{1},r}(\Omega).
	\label{emb9}
	\end{equation}
\end{theorem}

\begin{proof}
	Let 
	$f\in \dot{K}_{p,m}^{\alpha _{2},r}(\Omega ),$
	and choose $\theta \in (0,1)$ such that 
	\begin{equation*}
	\frac{1}{q}=\frac{\theta }{p}+\frac{1-\theta }{p^{\ast }}.
	\end{equation*}%
	By the interpolation inequality \eqref{Interpolation}, we obtain 
	\begin{equation*}
	\Vert f\Vert _{\dot{K}_{q}^{\alpha _{1},r}(\Omega )}\leq \Vert f\Vert _{\dot{%
			K}_{p}^{\alpha _{1},r}(\Omega )}^{\theta }\,\Vert f\Vert _{\dot{K}_{p^{\ast
		}}^{\alpha _{1},r}(\Omega )}^{1-\theta }.
	\end{equation*}%
	Applying Theorems \ref{embedingsq=p} and \ref{embeddingsfirst}, we deduce
	that 
	\begin{equation*}
	\Vert f\Vert _{\dot{K}_{q}^{\alpha _{1},r}(\Omega )}\lesssim \Vert f\Vert _{%
		\dot{K}_{p,m}^{\alpha _{2},r}(\Omega )}.
	\end{equation*}%
	Hence 
	\begin{equation*}
	\dot{K}_{p,m}^{\alpha _{2},r}(\Omega )\hookrightarrow \dot{K}_{q}^{\alpha
		_{1},r}(\Omega ).
	\end{equation*}
	
	The embedding \eqref{emb9} follows similarly by using the interpolation
	inequality \eqref{Interpolation1}. The proof is complete.
\end{proof}

\begin{remark}
	A natural question is whether the embeddings obtained in Theorems \ref%
	{embeddingsfirst} and \ref{embeddingsfirst copy(1)} remain valid when 
	\begin{equation*}
	q>p^{\ast}.
	\end{equation*}
	Example \ref{example1} shows that this is generally false. Indeed, one can
	choose 
	\begin{equation*}
	m-\frac{n}{p}-\alpha_{2} < \kappa < -\frac{n}{q}-\alpha_{1}
	\end{equation*}
	such that 
	\begin{equation*}
	g\in \dot{K}_{p,m}^{\alpha_{2},r}(\Omega) \quad \text{but} \quad g\notin 
	\dot{K}_{q}^{\alpha_{1},v}(\Omega),
	\end{equation*}
	and similarly, 
	\begin{equation*}
	g\in K_{p,m}^{\alpha_{2},r}(\Omega) \quad \text{but} \quad g\notin
	K_{q}^{\alpha_{1},v}(\Omega).
	\end{equation*}
\end{remark}

We now study the limiting case 
\begin{equation*}
\alpha_{2}=n-\frac{n}{p}
\end{equation*}
in Theorem \ref{embedingsq=p}.

\begin{theorem}
	\label{embedingsq=p copy(1)} Let $\Omega \subset \mathbb{R}^{n}$ be a domain
	satisfying the cone condition, and let $m\in \mathbb{N}$. Assume that 
	$1\le p<\infty, 1\le r<\infty,$
	and 
	\begin{equation*}
	-\frac{n}{p}<\alpha \le n-\frac{n}{p}, \qquad m>n-\frac{n}{p}-\alpha.
	\end{equation*}
	Then 
	\begin{equation}
	\dot{K}_{p,m}^{\,n-\frac{n}{p},1}(\Omega) \hookrightarrow \dot{K}%
	_{p}^{\alpha,r}(\Omega),  \label{limitcase}
	\end{equation}
	and 
	\begin{equation*}
	K_{p,m}^{\,n-\frac{n}{p},1}(\Omega) \hookrightarrow K_{p}^{\alpha,r}(\Omega)
	\end{equation*}
	hold.
\end{theorem}

\begin{proof}
	By similarity, it is sufficient to prove \eqref{limitcase}. In view of Theorem \ref{embedingsq=p}, it only remains to estimate $\Lambda_{2}$. More precisely, we need to estimate the terms $J_{1,k}$ and $J_{4,k}^{1}$ arising in the estimate of $\Lambda_{2}$. Clearly, 
	\begin{align*}
	J_{1,k}(x) &\lesssim \sum_{|\beta|\le m-1} \sum_{i=-\infty}^{j} 2^{i\left(n-%
		\frac{n}{p}\right)} \|(D^{\beta}f)\chi_{R_{i}\cap\Omega}\|_{p} \\
	&\lesssim \sum_{|\beta|\le m-1} \|D^{\beta}f\|_{\dot{K}_{p}^{\,n-\frac{n}{p}%
			,1}(\Omega)} \\
	&\lesssim \|f\|_{\dot{K}_{p,m}^{\,n-\frac{n}{p},1}(\Omega)}.
	\end{align*}
	Hence, 
	\begin{align*}
	\sum_{\substack{ k\in\mathbb{Z}  \\ \delta>2^{k+2}}} 2^{k\alpha r}
	\|J_{1,k}\chi_{R_{k}\cap\Omega}\|_{p}^{r} &\lesssim \sum_{\substack{ k\in%
			\mathbb{Z}  \\ \delta>2^{k+2}}} 2^{k\left(\alpha+\frac{n}{p}\right)r} \|f\|_{%
		\dot{K}_{p,m}^{\,n-\frac{n}{p},1}(\Omega)}^{r} \\
	&\lesssim \|f\|_{\dot{K}_{p,m}^{\,n-\frac{n}{p},1}(\Omega)}^{r},
	\end{align*}
	since $\alpha>-\frac{n}{p}$.
	
	Next, by H\"{o}lder's inequality, 
	\begin{align*}
	J_{4,k}^{1}(x) &\le 2^{k(m-n)} \sum_{|\beta|=m} \int_{|y|\le 2^{k+4}}
	|D^{\beta}f(y)|\chi_{\Omega}(y)\,dy \\
	&\lesssim 2^{k(m-n)} \sum_{|\beta|=m} \sum_{i=-\infty}^{k+4} 2^{i\left(n-%
		\frac{n}{p}\right)} \|(D^{\beta}f)\chi_{R_{i}\cap\Omega}\|_{p} \\
	&\lesssim 2^{k(m-n)} \|f\|_{\dot{K}_{p,m}^{\,n-\frac{n}{p},1}(\Omega)}.
	\end{align*}
	Therefore, 
	\begin{equation*}
	2^{k\alpha} \|J_{4,k}^{1}\chi_{R_{k}\cap\Omega}\|_{p} \lesssim
	2^{k\left(m-n+\alpha+\frac{n}{p}\right)} \|f\|_{\dot{K}_{p,m}^{\,n-\frac{n}{p%
			},1}(\Omega)}.
	\end{equation*}
	Consequently, 
	\begin{equation*}
	\sum_{\substack{ k\in\mathbb{Z}  \\ \delta>2^{k-2}}} 2^{k\alpha r}
	\|J_{4,k}^{1}\chi_{R_{k}\cap\Omega}\|_{p}^{r} \lesssim \|f\|_{\dot{K}%
		_{p,m}^{\,n-\frac{n}{p},1}(\Omega)}^{r},
	\end{equation*}
	where we have used the assumption 
	$m>n-\frac{n}{p}-\alpha.$

	The proof is complete.
\end{proof}

Next, we consider the case $q=\infty$.

\begin{theorem}
	\label{embeddingsq=infinity} Let $\Omega \subset \mathbb{R}^{n}$ be a domain
	satisfying the cone condition, and let $m\in \mathbb{N}$. Assume that 
	$1<p<\infty,  1\le r<\infty,0\le \alpha_{1}\le \alpha_{2}<n-\frac{n}{p}$
	and 
	\begin{equation*}
	m>\frac{n}{p}+\alpha_{2}-\alpha_{1}.
	\end{equation*}
	Then 
	\begin{equation}
	\dot{K}_{p,m}^{\alpha_{2},r}(\Omega) \hookrightarrow \dot{K}%
	_{\infty}^{\alpha_{1},v}(\Omega),  \label{q=infinity}
	\end{equation}
	and 
	\begin{equation*}
	K_{p,m}^{\alpha_{2},r}(\Omega) \hookrightarrow
	K_{\infty}^{\alpha_{1},v}(\Omega),
	\end{equation*}
	where 
	\begin{equation*}
	v= \left\{ 
	\begin{array}{ll}
	\infty, & \text{if } \alpha_{1}=0, \\[1mm] 
	r, & \text{if } \alpha_{1}>0.%
	\end{array}
	\right.
	\end{equation*}
\end{theorem}

\begin{proof}
	By similarity, it suffices to prove \eqref{q=infinity}. We also use Theorem %
	\ref{density}. Let 
	$f\in C^{\infty}(\Omega)\cap\dot{K}_{p,m}^{\alpha_{2},r}(\Omega)$, and let $\delta$ denote the dimension of the cone $C$.
	
	We only consider the case $\alpha_{1}>0$, since the case $\alpha_{1}=0$ is
	analogous. Write 
	\begin{equation*}
	\|f\|_{\dot{K}_{\infty}^{\alpha_{1},r}(\Omega)}^{r} = S_{1}+S_{2},
	\end{equation*}
	where 
	\begin{equation*}
	S_{1} = \sum_{\substack{ k\in\mathbb{Z}  \\ \delta\le 2^{k-2}}}
	2^{k\alpha_{1}r} \|f\chi_{R_{k}\cap\Omega}\|_{\infty}^{r},
	\end{equation*}
	and 
	\begin{equation*}
	S_{2} = \sum_{\substack{ k\in\mathbb{Z}  \\ \delta>2^{k-2}}}
	2^{k\alpha_{1}r} \|f\chi_{R_{k}\cap\Omega}\|_{\infty}^{r}.
	\end{equation*}
	
	\medskip
	
	\noindent \textbf{Estimate of $S_{1}$.} From Lemma \ref{Sobolev-rep1} and H%
	\"{o}lder's inequality, using the assumption $m>\frac{n}{p}$, we obtain 
	\begin{equation*}
	|f(x)| \lesssim \sum_{|\beta|\le m} \|(D^{\beta}f)\chi_{\widetilde{R}%
		_{k}\cap\Omega}\|_{p}
	\end{equation*}
	for every $x\in R_{k}\cap\Omega$ such that $ \delta\leq 2^{k-2}$, since 
	$C_{x,\delta}\subset \widetilde{R}_{k}\cap\Omega.$
	Hence, 
	\begin{align*}
	S_{1} &\lesssim \sum_{|\beta|\le m} \sum_{\substack{ k\in\mathbb{Z}  \\ %
			\delta\le 2^{k-2}}} 2^{k(\alpha_{1}-\alpha_{2})r} 2^{k\alpha_{2}r}
	\|(D^{\beta}f)\chi_{\widetilde{R}_{k}\cap\Omega}\|_{p}^{r} \\
	&\lesssim \|f\|_{\dot{K}_{p,m}^{\alpha_{2},r}(\Omega)}^{r},
	\end{align*}
	since $\alpha_{2}\ge \alpha_{1}$. Here we used estimate \eqref{emb5}.
	
	\medskip
	
	\noindent \textbf{Estimate of $S_{2}$.} Let $J_{1,k}$ and $J_{2,k}$ be
	defined as in Theorem \ref{embeddingsfirst}. We decompose 
	$J_{1,k}= J_{1,k}^{1}+J_{1,k}^{2},$
	where 
	\begin{align*}
	J_{1,k}^{1}(x) &= \sum_{|\beta|\le m-1} \int_{C_{x,2^{k+2}}}
	|D^{\beta}f(y)|\,dy, \qquad
	J_{1,k}^{2}(x) = \sum_{|\beta|\le m-1} \int_{B_{x,\delta,k}}
	|D^{\beta}f(y)|\,dy,
	\end{align*}
	
	where $x\in R_{k}\cap\Omega$. By H\"{o}lder's inequality, 
	\begin{align*}
	J_{1,k}^{1}(x) &\le \sum_{|\beta|\le m-1} \int_{|y|\le 2^{k+3}}
	|D^{\beta}f(y)|\chi_{\Omega}(y)\,dy \\
	&= \sum_{|\beta|\le m-1} \sum_{i=-\infty}^{k+3} \int_{\Omega}
	|D^{\beta}f(y)|\chi_{R_{i}}(y)\,dy \\
	&\lesssim \sum_{|\beta|\le m-1} \sum_{i=-\infty}^{k+4} 2^{i\left(n-\frac{n}{p%
		}-\alpha_{2}\right)} 2^{i\alpha_{2}}
	\|(D^{\beta}f)\chi_{R_{i}\cap\Omega}\|_{p} \\
	&\lesssim 2^{k\left(n-\frac{n}{p}-\alpha_{2}\right)} \sum_{|\beta|\le m-1}
	\sum_{i=-\infty}^{k+4} 2^{(i-k)\left(n-\frac{n}{p}-\alpha_{2}\right)}
	2^{i\alpha_{2}} \|(D^{\beta}f)\chi_{R_{i}\cap\Omega}\|_{p} \\
	&\lesssim 2^{k\left(n-\frac{n}{p}-\alpha_{2}\right)} \|f\|_{\dot{K}%
		_{p,m}^{\alpha_{2},r}(\Omega)},
	\end{align*}
	since $\alpha_{2}<n-\frac{n}{p}$. Therefore, 
	\begin{align*}
	\sum_{\substack{ k\in\mathbb{Z}  \\ \delta>2^{k-2}}} 2^{k\alpha_{1}r}
	\sup_{x\in R_{k}\cap\Omega} (J_{1,k}^{1}(x))^{r} \lesssim \|f\|_{\dot{K}%
		_{p,m}^{\alpha_{2},r}(\Omega)}^{r},
	\end{align*}
	because 
	\begin{equation*}
	\alpha_{1}+n-\frac{n}{p}-\alpha_{2}>0.
	\end{equation*}
	
	The estimate of $J_{1,k}^{2}$ follows from Lemma \ref{key1}.
	
	Next, we estimate $J_{2,k}$. We write 
	\begin{align*}
	J_{2,k}(x) &\lesssim \sum_{|\beta|\le m} \int_{C_{x,\delta}} \frac{%
		|D^{\beta}f(y)|}{|x-y|^{n-m}}\,dy \\
	&= \sum_{|\beta|\le m} \int_{C_{x,2^{k-2}}} \frac{|D^{\beta}f(y)|}{%
		|x-y|^{n-m}}\,dy 
	+ \sum_{|\beta|\le m} \int_{B_{x,\delta,k-4}} \frac{|D^{\beta}f(y)|}{%
		|x-y|^{n-m}}\,dy \\
	&=: P_{1,k}(x)+P_{2,k}(x).
	\end{align*}
	
	We have 
	\begin{align*}
	P_{1,k}(x) &\le c\sum_{|\beta|=m} 2^{k(\alpha_{2}-\alpha_{1})}
	\int_{C_{x,2^{k-2}}} \frac{|D^{\beta}f(y)|\chi_{\widetilde{R}%
			_{k}\cap\Omega}(y)} {|x-y|^{n-m-\alpha_{1}+\alpha_{2}}}\,dy \\
	&= c\,2^{k(\alpha_{2}-\alpha_{1})} \sum_{|\beta|=m}  |D^{\beta}f|\chi_{%
		\widetilde{R}_{k}\cap\Omega} \ast \chi_{\overline{B(0,\delta)}%
	}\omega_{m,\alpha_{1},\alpha_{2}}(x).
	\end{align*}
	
	By H\"{o}lder's inequality, 
	\begin{align*}
	\big\| |D^{\beta}f|\chi_{\widetilde{R}_{k}\cap\Omega} \ast \chi_{\overline{%
			B(0,\delta)}}\omega_{m,\alpha_{1},\alpha_{2}} \big\|_{\infty} \lesssim
	\|(D^{\beta}f)\chi_{\widetilde{R}_{k}\cap\Omega}\|_{p},
	\end{align*}
	since 
	$m>\frac{n}{p}+\alpha_{2}-\alpha_{1}.$
	Consequently, 
	\begin{equation*}
	\sum_{\substack{ k\in\mathbb{Z}  \\ \delta>2^{k-2}}} 2^{k\alpha_{1}r}
	\sup_{x\in R_{k}\cap\Omega} (P_{1,k}(x))^{r} \lesssim \|f\|_{\dot{K}%
		_{p,m}^{\alpha_{2},r}(\Omega)}^{r}.
	\end{equation*}
	
	Finally, we decompose $P_{2,k}$ as 
	\begin{equation}
	P_{2,k}(x) = \sum_{|\beta|\le m} \int_{B_{x,2^{k+2},k-4}} \frac{%
		|D^{\beta}f(y)|}{|x-y|^{n-m}}\,dy + \sum_{|\beta|\le m}
	\digamma_{k,\beta}(x).  \label{term1}
	\end{equation}
	The first term in \eqref{term1} can be estimated exactly as in the estimate
	of $J_{4,k}^{1}$ in Theorem \ref{embedingsq=p}, while the second term
	follows from Lemmas \ref{key1} and \ref{key1 copy(4)}. This completes the
	proof.
\end{proof}

\begin{remark}
	\label{embeddingsq=infinity copy(4)} Let $\Omega \subset \mathbb{R}^{n}$ be
	a domain satisfying the cone condition, and let $m\in \mathbb{N}$. Assume
	that $1<p<\infty $. As a consequence of Theorem \ref{embeddingsq=infinity},
	we obtain 
	\begin{equation*}
	\dot{K}_{p,m}^{\,n-\frac{n}{p},1}(\Omega )\hookrightarrow L^{\infty }(\Omega
	)
	\end{equation*}%
	and 
	\begin{equation*}
	K_{p,m}^{\,n-\frac{n}{p},1}(\Omega )\hookrightarrow L^{\infty }(\Omega )
	\end{equation*}%
	whenever $m>n.$
\end{remark}

For an arbitrary nonempty set $\Omega \subset \mathbb{R}^{n}$, we denote by $%
C_{b}(\Omega)$ the Banach space of all bounded continuous functions on $%
\Omega$, equipped with the norm 
\begin{equation*}
\|f\|_{C_{b}(\Omega)} = \sup_{x\in\Omega}|f(x)|.
\end{equation*}

Combining Theorems \ref{density} and \ref{embeddingsq=infinity}, Remark \ref%
{embeddingsq=infinity copy(4)}, and arguing as in \cite[Theorem 4.12]%
{AdamsFournier03}, we obtain the following important result.

\begin{theorem}
	\label{embeddingsq=infinity copy(3)} Let $\Omega \subset \mathbb{R}^{n}$ be
	a domain satisfying the cone condition, and let $m\in \mathbb{N}$. Assume
	that 
	$1<p<\infty, 1\le r<\infty,0\le \alpha \le n-\frac{n}{p},$
	and 
	$m>\frac{n}{p}+\alpha.$
	Then 
	\begin{equation*}
	\dot{K}_{p,m}^{\alpha,v}(\Omega) \hookrightarrow C_{b}(\Omega)
	\end{equation*}
	and 
	\begin{equation*}
	K_{p,m}^{\alpha,v}(\Omega) \hookrightarrow C_{b}(\Omega),
	\end{equation*}
	where 
	\begin{equation*}
	v= \left\{ 
	\begin{array}{ll}
	1, & \text{if } \alpha=n-\dfrac{n}{p}, \\[2mm] 
	r, & \text{otherwise.}%
	\end{array}
	\right.
	\end{equation*}
\end{theorem}

We now study the  case 
\begin{equation*}
m\ge \frac{n}{p}+\alpha_{2}-\alpha_{1}
\end{equation*}
in Theorem \ref{embeddingsfirst copy(1)}.

\begin{theorem}
	\label{embedingsp<q} Let $\Omega \subset \mathbb{R}^{n}$ be a domain
	satisfying the cone condition, and let $m\in \mathbb{N}$. Assume that 
	$1<p<q<\infty, 1\le r<\infty,$
	$-\frac{n}{q}<\alpha_{1}\le \alpha_{2}<n-\frac{n}{p},$
	and 
	\begin{equation*}
	m\ge \frac{n}{p}+\alpha_{2}-\alpha_{1}.
	\end{equation*}
	Then 
	\begin{equation}
	\dot{K}_{p,m}^{\alpha_{2},r}(\Omega) \hookrightarrow \dot{K}%
	_{q}^{\alpha_{1},r}(\Omega),  \label{emb10}
	\end{equation}
	and 
	\begin{equation*}
	K_{p,m}^{\alpha_{2},r}(\Omega) \hookrightarrow K_{q}^{\alpha_{1},r}(\Omega)
	\end{equation*}
	hold.
\end{theorem}

\begin{proof}
	By similarity, it suffices to prove \eqref{emb10}. Let 
	$f\in C^{\infty}(\Omega)\cap \dot{K}_{p,m}^{\alpha_{2},r}(\Omega)$,
	and let $\delta$ denote the dimension of the cone $C$. We write 
	\begin{equation*}
	\|f\|_{\dot{K}_{q}^{\alpha_{1},r}(\Omega)}^{r} = \Delta_{1}+\Delta_{2},
	\end{equation*}
	where 
	\begin{equation*}
	\Delta_{1} = \sum_{\substack{ k\in\mathbb{Z}  \\ \delta\le 2^{k-2}}}
	2^{k\alpha_{1}r} \|f\chi_{R_{k}\cap\Omega}\|_{q}^{r},\quad\Delta_{2} = \sum_{\substack{ k\in\mathbb{Z}  \\ \delta>2^{k-2}}}
	2^{k\alpha_{1}r} \|f\chi_{R_{k}\cap\Omega}\|_{q}^{r}.
	\end{equation*}

	Let $J_{1,k}$ and $J_{2,k}$ be defined as in Theorem \ref{embeddingsfirst}.
	
	\medskip
	
	\noindent \textbf{Estimate of $\Delta_{1}$.} Let $x\in R_{k}\cap\Omega$,  $k\in\mathbb{Z}$,  $\delta\le 2^{k-2}$. Recall that 
	\begin{equation*}
	J_{1,k}(x) \le \sum_{|\beta|\le m-1} |D^{\beta}f| \chi_{\widetilde{R}%
		_{k}\cap\Omega} \ast \chi_{\overline{B(0,\delta)}}(x),
	\end{equation*}
	and 
	\begin{equation*}
	J_{2,k}(x) \lesssim 2^{k(\alpha_{2}-\alpha_{1})} \sum_{|\beta|=m}
	|D^{\beta}f| \chi_{\widetilde{R}_{k}\cap\Omega} \ast \chi_{\overline{%
			B(0,\delta)}} \omega_{m,\alpha_{1},\alpha_{2}}(x).
	\end{equation*}
	
	Let $g$ be a measurable function on $\Omega$, and define 
	\begin{equation*}
	T(g)(x) = |g| \ast \chi_{\overline{B(0,\delta)}}
	\varpi_{m,\alpha_{1},\alpha_{2}}(x),
	\end{equation*}
	where 
	\begin{equation*}
	\varpi_{m,\alpha_{1},\alpha_{2}} \equiv \omega_{m,\alpha_{1},\alpha_{2}} \quad 
	\text{or} \quad \varpi_{m,\alpha_{1},\alpha_{2}}\equiv 1.
	\end{equation*}
	
	Choose  $ 1<p_{1}<p<p_{2}<\infty $
	such that 
	\begin{equation*}
	\frac{1}{p} = \frac{\theta}{p_{1}} + \frac{1-\theta}{p_{2}}, \qquad \theta=%
	\frac{p_{1}}{q}<1.
	\end{equation*}
	
	By Young's inequality, and using the fact that 
	$m>\frac{n}{p_{2}}+\alpha_{2}-\alpha_{1}$
	when 
	$\varpi_{m,\alpha_{1},\alpha_{2}} = \omega_{m,\alpha_{1},\alpha_{2}},$
	we obtain 
	\begin{align*}
	\|T(g)\|_{\infty} &\lesssim \|\chi_{\overline{B(0,\delta)}}
	\omega_{m,\alpha_{1},\alpha_{2}}\|_{p_{2}^{\prime }} \|g\|_{p_{2}} \\
	&\lesssim \|g\|_{p_{2}}.
	\end{align*}
	Hence $T$ is bounded from $L^{p_{2}}(\mathbb{R}^{n})$ into $L^{\infty}(%
	\mathbb{R}^{n})$. Again by Young's inequality, together with 
	$m>\alpha_{2}-\alpha_{1},$
	we obtain 
	\begin{align*}
	\|T(g)\|_{p_{1}} &\lesssim \|\chi_{\overline{B(0,\delta)}}
	\omega_{m,\alpha_{1},\alpha_{2}}\|_{1} \|g\|_{p_{1}} \\
	&\lesssim \|g\|_{p_{1}}.
	\end{align*}
	Therefore $T$ is bounded from $L^{p_{1}}(\mathbb{R}^{n})$ into itself. By
	the Marcinkiewicz interpolation theorem, $T$ is bounded from $L^{p}(\mathbb{R%
	}^{n})$ into $L^{q}(\mathbb{R}^{n})$. Consequently, 
	\begin{equation*}
	\|J_{1,k}\chi_{R_{k}\cap\Omega}\|_{q} \lesssim \sum_{|\beta|\le m-1}
	\|(D^{\beta}f)\chi_{\widetilde{R}_{k}\cap\Omega}\|_{p},
	\end{equation*}
	and 
	\begin{equation*}
	\|J_{2,k}\chi_{R_{k}\cap\Omega}\|_{q} \lesssim 2^{k(\alpha_{2}-\alpha_{1})}
	\sum_{|\beta|=m} \|(D^{\beta}f)\chi_{\widetilde{R}_{k}\cap\Omega}\|_{p}.
	\end{equation*}
	Combining these estimates yields 
	\begin{equation*}
	\Delta_{1} \lesssim \|f\|_{\dot{K}_{p,m}^{\alpha_{2},r}(\Omega)}^{r}.
	\end{equation*}
	
	\medskip
	
	\noindent \textbf{Estimate of $\Delta_{2}$.} From the proof of Theorem \ref%
	{embeddingsfirst}, we already know that 
	\begin{equation*}
	\sum_{\substack{ k\in\mathbb{Z}  \\ \delta>2^{k+2}}} 2^{k\alpha_{1}r}
	\|J_{1,k}\chi_{R_{k}\cap\Omega}\|_{q}^{r} \lesssim \|f\|_{\dot{K}%
		_{p,m}^{\alpha_{2},r}(\Omega)}^{r},
	\end{equation*}
	since 
	$\alpha_{2}<n-\frac{n}{p} $  and $  \alpha_{1}>-\frac{n}{q}.$
	
	Let $J_{3,k}$ and $J_{4,k}$ be defined as in Theorem \ref{embedingsq=p}.
	Recall that 
	$J_{2,k}=J_{3,k}+J_{4,k}.$
	Since 
	$C_{x,2^{k-2}} \subset \widetilde{R}_{k}\cap\Omega,$
	we obtain 
	\begin{equation*}
	J_{3,k}(x) \lesssim 2^{k(\alpha_{2}-\alpha_{1})} \sum_{|\beta|=m}
	|D^{\beta}f| \chi_{\widetilde{R}_{k}\cap\Omega} \ast \chi_{\overline{%
			B(0,\delta)}} \omega_{m,\alpha_{1},\alpha_{2}}(x).
	\end{equation*}
	
	Arguing exactly as in the estimate of $\Delta_{1}$, we get 
	\begin{equation*}
	\big\| |D^{\beta}f| \chi_{\widetilde{R}_{k}\cap\Omega} \ast \chi_{\overline{%
			B(0,\delta)}} \omega_{m,\alpha_{1},\alpha_{2}} \big\|_{q} \lesssim
	\|(D^{\beta}f)\chi_{\widetilde{R}_{k}\cap\Omega}\|_{p}, \qquad |\beta|=m.
	\end{equation*}
	Hence, 
	\begin{equation}
	\sum_{\substack{ k\in\mathbb{Z}  \\ \delta>2^{k-2}}} 2^{k\alpha_{1}r}
	\|J_{3,k}\chi_{R_{k}\cap\Omega}\|_{q}^{r} \lesssim \|f\|_{\dot{K}%
		_{p,m}^{\alpha_{2},r}(\Omega)}^{r}.  \label{lambda21}
	\end{equation}
	
	Now write 
	$J_{4,k} = J_{4,k}^{1}+J_{4,k}^{2},$
	and recall that
	\begin{equation*}
	J_{4,k}^{1} \lesssim 2^{k(m-n)} \sum_{|\beta|=m} \sum_{i=-\infty}^{k+4}
	2^{i\left(n-\frac{n}{p}-\alpha_{2}\right)} 2^{i\alpha_{2}}
	\|(D^{\beta}f)\chi_{R_{i}\cap\Omega}\|_{p}.
	\end{equation*}
	
	Since $\alpha_{2}<n-\frac{n}{p},$ we obtain 
	\begin{align*}
	2^{k\alpha_{1}} \|J_{4,k}^{1}\chi_{R_{k}\cap\Omega}\|_{q} &\lesssim
	2^{k\left(m+\alpha_{1}-\alpha_{2}+\frac{n}{q}-\frac{n}{p}\right)} 
	\sum_{|\beta|=m}\sup_{i\in\mathbb{Z}} \left( 2^{i\alpha_{2}}
	\|(D^{\beta}f)\chi_{R_{i}\cap\Omega}\|_{p} \right) \\
	&\lesssim 2^{k\left(m+\alpha_{1}-\alpha_{2}+\frac{n}{q}-\frac{n}{p}\right)}
	\|f\|_{\dot{K}_{p,m}^{\alpha_{2},r}(\Omega)}.
	\end{align*}
	Therefore, 
	\begin{equation*}
	\sum_{\substack{ k\in\mathbb{Z}  \\ \delta>2^{k-2}}} 2^{k\alpha_{1}r}
	\|J_{4,k}^{1}\chi_{R_{k}\cap\Omega}\|_{q}^{r} \lesssim \|f\|_{\dot{K}%
		_{p,m}^{\alpha_{2},r}(\Omega)}^{r},
	\end{equation*}
	where we used the assumption
	$m> \frac{n}{p} + \alpha_{2} - \alpha_{1} - \frac{n}{q}.$
	
	Finally, Lemmas \ref{key1} and \ref{key1 copy(4)} imply that 
	\begin{equation*}
	\sum_{\substack{ k\in\mathbb{Z}  \\ \delta>2^{k-2}}} 2^{k\alpha_{1}r}
	\|J_{4,k}^{2}\chi_{R_{k}\cap\Omega}\|_{q}^{r} \lesssim \|f\|_{\dot{K}%
		_{p,m}^{\alpha_{2},r}(\Omega)}^{r}.
	\end{equation*}
	This completes the proof.
\end{proof}

\begin{remark}
	When 
	$m>\frac{n}{p}+\alpha _{2}-\alpha _{1},$
	Theorem \ref{embedingsp<q} can also be derived directly from Theorems \ref%
	{embedingsq=p} and \ref{embeddingsq=infinity}. Indeed, let 
	$f\in \dot{K}_{p,m}^{\alpha _{2},r}(\Omega ),$
	and set $\theta =\frac{p}{q}.$
	By the interpolation inequality \eqref{Interpolation}, we obtain 
	\begin{equation*}
	\Vert f\Vert _{\dot{K}_{q}^{\alpha _{1},r}(\Omega )}\leq \Vert f\Vert _{\dot{%
			K}_{p}^{\alpha _{1},r}(\Omega )}^{\theta }\,\Vert f\Vert _{\dot{K}_{\infty
		}^{\alpha _{1},r}(\Omega )}^{1-\theta }.
	\end{equation*}%
	Applying Theorems \ref{embedingsq=p} and \ref{embeddingsq=infinity}, we
	deduce that 
	\begin{equation*}
	\Vert f\Vert _{\dot{K}_{q}^{\alpha _{1},r}(\Omega )}\lesssim \Vert f\Vert _{%
		\dot{K}_{p,m}^{\alpha _{2},r}(\Omega )}.
	\end{equation*}%
	Hence, 
	\begin{equation*}
	\dot{K}_{p,m}^{\alpha _{2},r}(\Omega )\hookrightarrow \dot{K}_{q}^{\alpha
		_{1},r}(\Omega ).
	\end{equation*}
	
	The same argument applies to the nonhomogeneous Herz spaces.
\end{remark}

As a consequence of Theorem \ref{embedingsp<q}, we obtain the following
result.

\begin{theorem}
	\label{embedingsp<q copy(1)} Let $\Omega \subset \mathbb{R}^{n}$ be a domain
	satisfying the cone condition and let $m\in \mathbb{N}_{0}$. Let 
	$1<p<q<\infty,1\leq r<\infty,$
	and assume that 
	\begin{equation*}
	-\frac{n}{q}<\alpha \leq n-\frac{n}{p}, \qquad m\geq n-\alpha.
	\end{equation*}
	Then 
	\begin{equation*}
	\dot{K}_{p,m}^{n-\frac{n}{p},1}(\Omega) \hookrightarrow \dot{K}%
	_{q}^{\alpha,r}(\Omega)
	\end{equation*}
	and 
	\begin{equation*}
	K_{p,m}^{n-\frac{n}{p},1}(\Omega) \hookrightarrow K_{q}^{\alpha,r}(\Omega).
	\end{equation*}
	In particular, 
	\begin{equation*}
	\dot{K}_{p,m}^{n-\frac{n}{p},1}(\Omega) \hookrightarrow L^{q}(\Omega),\qquad m\geq n
	\end{equation*}
	and 
	\begin{equation*}
	K_{p,m}^{n-\frac{n}{p},1}(\Omega) \hookrightarrow L^{q}(\Omega), \qquad
	m\geq n.
	\end{equation*}
\end{theorem}

In the previous results, the case $q<p$ has not yet been considered. The
following theorem provides a positive answer in this situation.

\begin{theorem}
	\label{embedingsq<p} Let $\Omega \subset \mathbb{R}^{n}$ be a domain
	satisfying the cone condition and let $m\in \mathbb{N}$. Let 
	$1\leq q<p<\infty,  1\leq r<\infty,   \alpha_{1}>-\frac{n}{q},$
	and assume that 
	\begin{equation*}
	\alpha_{1}+\frac{n}{q}-\frac{n}{p} \leq \alpha_{2} < n-\frac{n}{p},
	\end{equation*}
	and 
	\begin{equation*}
	m\geq \frac{n}{p}-\frac{n}{q}+\alpha_{2}-\alpha_{1}.
	\end{equation*}
	Then 
	\begin{equation*}
	\dot{K}_{p,m}^{\alpha_{2},r}(\Omega) \hookrightarrow \dot{K}%
	_{q}^{\alpha_{1},r}(\Omega)
	\end{equation*}
	and 
	\begin{equation*}
	K_{p,m}^{\alpha_{2},r}(\Omega) \hookrightarrow K_{q}^{\alpha_{1},r}(\Omega)
	\end{equation*}
	hold.
\end{theorem}

\begin{proof}
	We use Theorem \ref{density}. Let 
	$f\in C^{\infty}(\Omega)\cap \dot{K}_{p,m}^{\alpha_{2},r}(\Omega)$, and let $\delta$ denote the dimension of the cone $C$. Set 
	\begin{equation*}
	\Gamma_{1} = \sum_{\substack{ k\in\mathbb{Z}  \\ \delta\leq 2^{k-2}}}
	2^{k\alpha_{1}r} \|f\chi_{R_{k}\cap\Omega}\|_{q}^{\,r},\qquad \Gamma_{2} = \sum_{\substack{ k\in\mathbb{Z}  \\ \delta>2^{k-2}}}
	2^{k\alpha_{1}r} \|f\chi_{R_{k}\cap\Omega}\|_{q}^{\,r}.
	\end{equation*}

	Lemma \ref{Sobolev-rep1} yields 
	\begin{align*}
	|f(x)| \lesssim\, & \sum_{|\beta|\leq m-1} |D^{\beta}f| \chi_{\tilde{R}%
		_{k}\cap\Omega} * \chi_{\overline{B(0,\delta)}}(x) 
	+ \sum_{|\beta|=m} |D^{\beta}f| \chi_{\tilde{R}_{k}\cap\Omega} * \chi_{%
		\overline{B(0,\delta)}} \omega_{m,0,0}(x)
	\end{align*}
	for every 
	$
	x\in R_{k}\cap\Omega,  k\in\mathbb{Z},  \delta\leq 2^{k-2},
	$
	where the implicit constant is independent of $k$. By Young's inequality, 
	\begin{align*}
	& \big\| |D^{\beta}f| \chi_{\tilde{R}_{k}\cap\Omega} * \chi_{\overline{%
			B(0,\delta)}} \omega_{m,0,0} \big\|_{p} 
	+ \big\| |D^{\beta}f| \chi_{\tilde{R}_{k}\cap\Omega} * \chi_{%
		\overline{B(0,\delta)}} \big\|_{p} \\
	\lesssim\, & \big\| (D^{\beta}f) \chi_{\tilde{R}_{k}\cap\Omega} \big\|_{p}.
	\end{align*}
	
	Using H\"{o}lder's inequality, we obtain 
	\begin{align*}
	\Gamma_{1} \lesssim\, & \sum_{|\beta|\leq m} \sum_{\substack{ k\in\mathbb{Z} 
			\\ \delta\leq 2^{k-2}}} 2^{(\alpha_{1}+\frac{n}{q}-\frac{n}{p})kr}
	\|(D^{\beta}f)\chi_{R_{k}\cap\Omega}\|_{p}^{\,r} \\
	\lesssim\, & \|f\|_{\dot{K}_{p,m}^{\alpha_{2},r}(\Omega)}^{r},
	\end{align*}
	since 
	$\alpha_{2}+\frac{n}{p} \geq \alpha_{1}+\frac{n}{q}.$
	
	We now estimate $\Gamma_{2}$. Let $J_{1,k}$ and $J_{2,k}$ be defined as in
	Theorem \ref{embeddingsfirst}. Recall that 
	\begin{equation*}
	\sum_{\substack{ k\in\mathbb{Z}  \\ \delta>2^{k+2}}} 2^{k\alpha_{1}r}
	\|J_{1,k}\chi_{R_{k}\cap\Omega}\|_{q}^{\,r} \lesssim \|f\|_{\dot{K}%
		_{p,m}^{\alpha_{2},r}(\Omega)}^{r},
	\end{equation*}
	because 
	$\alpha_{1}>-\frac{n}{q} $ and $ \alpha_{2}<n-\frac{n}{p}.$
	
	Moreover, 
	\begin{align*}
	J_{2,k}(x) =& \sum_{|\beta|=m} \int_{C_{x,2^{k-2}}} \frac{|D^{\beta}f(y)|}{%
		|x-y|^{n-m}} \,dy 
	+ \sum_{|\beta|=m} \int_{B_{x,\delta,k-4}} \frac{|D^{\beta}f(y)|}{%
		|x-y|^{n-m}} \,dy \\
	=:& J_{3,k}(x)+J_{4,k}(x).
	\end{align*}
	
	Observe that 
	\begin{align*}
	J_{3,k}(x) \leq\, & \sum_{|\beta|=m} \sum_{l=-\infty}^{k-2}
	\int_{2^{l-1}\leq |x-y|\leq 2^{l}} \frac{|D^{\beta}f(y)|}{|x-y|^{n-m}} \,dy
	\\
	\leq\, & \sum_{|\beta|=m} \sum_{l=-\infty}^{k-2} 2^{l(m-n)}
	\int_{2^{l-1}\leq |x-y|\leq 2^{l}} |D^{\beta}f(y)| \,dy \\
	\lesssim\, & 2^{km} \sum_{|\beta|=m} \mathcal{M} \big( (D^{\beta}f) \chi_{%
		\tilde{R}_{k}\cap\Omega} \big)(x),
	\end{align*}
	for every $x\in R_{k}\cap\Omega, k\in\mathbb{Z}, \delta>2^{k-2}$, where the
	implicit constant is independent of $k$.
	By H\"{o}lder's inequality and Lemma \ref{Maximal-Inq}, 
	\begin{align*}
	& \sum_{|\beta|=m} \sum_{\substack{ k\in\mathbb{Z}  \\ \delta>2^{k-2}}}
	2^{\alpha_{1}kr} \|J_{3,k}\chi_{R_{k}\cap\Omega}\|_{q}^{\,r} \\
	\lesssim\, & \sum_{|\beta|=m} \sum_{\substack{ k\in\mathbb{Z}  \\ %
			\delta>2^{k-2}}} 2^{(m-\frac{n}{p}+\alpha_{1}+\frac{n}{q})kr}
	\|(D^{\beta}f)\chi_{\tilde{R}_{k}\cap\Omega}\|_{p}^{\,r} \\
	\lesssim\, & \|f\|_{\dot{K}_{p,m}^{\alpha_{2},r}(\Omega)}^{r},
	\end{align*}
	because 
	$m\geq \frac{n}{p}-\frac{n}{q} +\alpha_{2}-\alpha_{1}.$
	
	Next, 
	\begin{align*}
	J_{4,k}(x) =& \sum_{|\beta|=m} \int_{B_{x,2^{k+2},k-4}} \frac{|D^{\beta}f(y)|%
	}{|x-y|^{n-m}} \,dy \\
	& + \sum_{|\beta|=m} \digamma_{k,\beta}(x) \\
	=:& J_{4,k}^{1}(x)+J_{4,k}^{2}(x).
	\end{align*}
	
	By H\"{o}lder's inequality, 
	\begin{align*}
	J_{4,k}^{1}(x) \leq\, & 2^{k(m-n)} \sum_{|\beta|=m} \sum_{i=-\infty}^{k+4}
	\int_{\Omega} |D^{\beta}f(y)| \chi_{R_{i}}(y) \,dy \\
	\lesssim\, & 2^{k(m-n)} \sum_{|\beta|=m} \sum_{i=-\infty}^{k+4} 2^{i(n-\frac{%
			n}{p}-\alpha_{2})} 2^{i\alpha_{2}}
	\|(D^{\beta}f)\chi_{R_{i}\cap\Omega}\|_{p}.
	\end{align*}
	
	Since 
	$m+\alpha_{1}+\frac{n}{q} \geq \frac{n}{p}+\alpha_{2},
	$
	we have $2^{k(m+\alpha_{1}+\frac{n}{q})} \lesssim 2^{k(\frac{n}{p}%
		+\alpha_{2})}$ for all $k\in\mathbb{Z}$ such that $\delta>2^{k-2}$.
	Therefore, 
	\begin{align*}
	2^{k\alpha_{1}} \|J_{4,k}^{1}\chi_{R_{k}\cap\Omega}\|_{q} \lesssim\, &
	\sum_{|\beta|=m} \sum_{i=-\infty}^{k+4} 2^{(i-k)(n-\frac{n}{p}-\alpha_{2})}
	2^{i\alpha_{2}} \|(D^{\beta}f)\chi_{R_{i}\cap\Omega}\|_{p}.
	\end{align*}
	
	Using Lemma \ref{lq-inequality} and the condition $\alpha_{2}<n-\frac{n}{p}$%
	, we obtain 
	\begin{equation*}
	\sum_{\substack{ k\in\mathbb{Z}  \\ \delta>2^{k-2}}} 2^{k\alpha_{1}r}
	\|J_{4,k}^{1}\chi_{R_{k}\cap\Omega}\|_{q}^{\,r} \lesssim \|f\|_{\dot{K}%
		_{p,m}^{\alpha_{2},r}(\Omega)}^{r}.
	\end{equation*}
	
	The estimate for $J_{4,k}^{2}$ follows from Lemma \ref{key1}. This completes
	the proof.
\end{proof}

\begin{remark}
	Theorems \ref{embedingsq=p}, \ref{embeddingsq=infinity}, and \ref%
	{embedingsq<p} were established in \cite[Theorem 3.14]{Dr-Sobolev} under the
	additional assumption 
	$\alpha_{2}<m-\frac{n}{p}.$
\end{remark}

\begin{remark}
	The results established in this paper contribute to a deeper understanding of Herz-type Sobolev spaces on domains and provide useful analytical tools for future research.
	
	Possible directions for further investigation include the study of pointwise multiplication, the Rellich--Kondrachov theorem extension operators, capacities, interpolation inequalities, applications to partial differential equations, and extensions to variable exponent and other generalized function spaces.
\end{remark}


\end{document}